\documentclass[10pt,reqno]{amsart}

\usepackage{amsmath,amssymb,amsthm}
\usepackage{mathrsfs,amsfonts,functan,extarrows,mathtools}
\usepackage[top=2.0cm,bottom=2.0cm,left=3.2cm,right=3.1cm]{geometry}
\usepackage[colorlinks]{hyperref}

\usepackage{marginnote}
\usepackage{xcolor}

\newtheorem{theorem}{Theorem}[section]
\newtheorem{proposition}{Proposition}[section]

\newtheorem{remark}{Remark}[section]
\newtheorem{lemma}{Lemma}[section]

\numberwithin{equation}{section}
\allowdisplaybreaks

\def\d{\textrm{\,d}}
\newcommand{\p}{{\partial}}
\newcommand{\eps}{\varepsilon}
\newcommand{\grad}{\nabla_{\!x}}

\def\no{\nonumber}
\def\<{\langle}
\def\>{\rangle}
\def\Int{\mathrm{int}}
\def\b{\mathrm{b}}

\newcommand{\A}{{\mathcal{A}}}
\newcommand{\F}{{\mathcal{F}}}
\newcommand{\G}{{\mathcal{G}}}

\newcommand{\abs}[1]{\left|#1\right|}
\newcommand{\abss}[1]{\left\|#1\right\|}
\newcommand{\skp}[2]{\left\langle #1,\, #2 \right\rangle}
\newcommand{\skpt}[2]{\langle #1,\, #2 \rangle}

\newcommand{\alert}[1]{{\color{red}\bf\large #1}} 

\begin{document}
\title[Boundary Layer]{Incompressible Limit of isentropic Navier-Stokes equations with Navier-slip boundary}
\author[L.-J. Xiong]{Linjie Xiong}
\address[Linjie Xiong]{\newline Institute of Mathematics, Hunan University, Changsha, 410082, Hunan, P. R. China}
\email{xlj@hnu.edu.cn}

\thanks{The research was supported by NSFC (Grant No. 11501187) and Fundamental Research Funds for the Central Universities. 
}

\maketitle

\begin{abstract}
This paper concerns the low Mach number limit of weak solutions to the compressible Navier-Stokes equations for isentropic fluids in a bounded domain with a Navier-slip boundary condition. In \cite{DGLM99}, it has been proved that if the velocity is imposed the homogeneous Dirichlet boundary condition, as the Mach number goes to 0, the velocity of the compressible flow converges strongly in $L^2$ under the geometrical assumption (H) on the domain. We justify the same strong convergence when the slip length in the Navier condition is the reciprocal of the square root of the Mach number.
\end{abstract}


\section{Introduction}
This paper is devoted to the incompressible limit for weak solutions to the compressible isentropic Navier-Stokes equations. This has been an active research field since the work of Lions-Masmoudi \cite{LM} in which various cases were considered, namely the cases of periodic flow, a flow in the whole space, or in a bounded domain with a Dirichlet boundary conditions with ``well-prepared" initial data (the density is close to constant and the velocity is close to incompressible), or with other boundary condition such as Navier condition with the constant slip length. The key idea of \cite{LM} is to employ the so-called ``group" method developed by Schochet \cite{Schochet94} and Grenier \cite{Grenier97} to show that the fast acoustic waves does not affect the incompressible limit, thus justified the weak convergence.

For the whole space case, Desjardins and Grenier \cite{DG4} investigated the dispersive properties of the acoustic waves and used the Strichartz estimate to show that the acoustic waves vanish as the Mach number goes to zero, thus justified the strong convergence.

In \cite{DGLM99}, Desjardins, Grenier, Lions and Masmoudi found a striking new phenomena. In the case of a viscous flow in a bounded domain with the usual Dirichlet boundary condition, under a generic assumption, they showed that the acoustic waves are instantaneously (asymptotically) damped, due to the formation of a thin boundary layer. This layer dissipated the energy carried by the waves. As a consequence, they obtained the strong $L^2$ convergence.

In this paper, we study the low Mach number limit of the Navier-Stokes equations for the isentropic compressible fluids. If $\eps$ denotes the Mach number which is the ratio of the macroscopic velocity and the sound speed, the compressible Navier-Stokes equations can be written as the following non-dimensional form in which we emphasize the dependence on $\eps$ of the density $\rho^\eps(x,t)$ and the velocity $\mathbf{u}^\eps(x,t)$. Here $x\in\Omega$, $t\in (0,\infty)$, and $\Omega$ is a bounded $C^2$ domain  of $\mathbb{R}^d$ with $d=2$ or $3$.
	\begin{align}\label{eq:CNS} 
	  \begin{cases}
	  	\partial_t \rho^\eps + {\rm div} (\rho^\eps \mathbf{u}^\eps) =0, \quad \text{in}\quad\! \Omega \times (0,\infty), \\[2pt]
	  	\partial_t (\rho^\eps \mathbf{u}^\eps) + {\rm div} (\rho^\eps \mathbf{u}^\eps \otimes \mathbf{u}^\eps) + \frac{\nabla_x (\rho^\eps)^\gamma}{\gamma \eps^2} = {\rm div} \Sigma (\mathbf{u}^\eps), \quad \text{in}\quad\! \Omega \times (0,\infty)\,,
	  \end{cases}
	\end{align}
with the initial data:
\begin{equation}\label{initial}
 \rho^{\eps}|_{t=0}=\rho_0^{\eps}\geq0,\quad  \rho^{\eps}\mathbf{u}^{\eps}|_{t=0}=m_0^{\eps} \quad  \text{in}~\Omega\,.
\end{equation}

 Here $\Sigma (\mathbf{u}) = \nu \sigma(\mathbf{u}) + \xi {\rm div} \mathbf{u} \mathrm{I}$ is the viscous stress tensor, where $\sigma(\mathbf{u}) = \nabla_x \mathbf{u} + \nabla_x^\top \mathbf{u}$ is the deformation tensor, and $\nu, \xi$ are the Lam\'{e} viscosity coefficients with $\nu\geq 0, \xi + \frac{2}{d}\nu \geq 0$. In this paper, we consider the system \eqref{eq:CNS} equipped with the Navier-slip boundary conditions:
	\begin{align}\label{eq:Robin BC}
	  \begin{cases}
	  \mathbf{u}^\eps \cdot \mathrm{n} =0, &\text{ on }\partial \Omega \times (0, \infty),
	  \\[2pt]
	  	\nu [\sigma (\mathbf{u}^\eps) \cdot \mathrm{n}]^\tau + \chi \sqrt\eps^{-1} (\mathbf{u}^\eps)^\tau =0, &\text{ on } \partial \Omega \times (0 ,\infty),
	  \end{cases}
	\end{align}
where $\mathrm{n}(x)$ is the outer normal vector at $x\in \partial\Omega$, and $\mathbf{u}^\tau(x)$ means the tangential part of the vector $\mathbf{u}$ at $x\in\partial\Omega$, and furthermore, $\chi > 0$ is a constant.

In this paper, we work in the context of weak solutions of \eqref{eq:CNS} for which a complete existence theory have been developed. If the initial data satisfies the following bound:
\begin{equation}\label{initialbound}
E_0^{\eps}=\int_{\Omega}\left[\pi_0^{\eps}+\frac{|m_0^{\eps}|^2}{2\rho_0^{\eps}}\right]\,\mathrm{d}x\leq C,
\end{equation}
where $C$ denotes various positive constants independent of $\eps$ and $\pi_0^\eps = \frac{(\rho_0^\eps)^\gamma-1-\gamma(\rho_0^\eps-1)}{\eps^2 \gamma (\gamma-1)}$.
Then the existence of global weak solutions to the above system (for a fixed $\eps>0$) was obtained by P.-L. Lions in \cite{L} under the condition:
$\gamma\geq\gamma_0$, where $\gamma_0=\frac{3}{2}$ if $d=2$ and $\gamma_0=\frac{9}{5}$ if $d=3$. Later, the adiabatic constant $\gamma$ is subjected to the constraint $\gamma>\frac{3}{2}$ by Feireisl, Novotny and Petzeltova in \cite{FNP} in three space dimensions.

The definition of weak solutions of \eqref{eq:CNS} is the following: $\rho_{\eps}\in L^{\infty}(0,\infty;L^{\gamma}(\Omega))$, $\mathbf{u}^{\eps}\in L^2(0,\infty;H^1(\Omega))$ solve \eqref{eq:CNS},\eqref{initial},\eqref{eq:Robin BC} in the sense of distributions. In addition, $\rho^{\eps}\in C([0,\infty);$ $L^p(\Omega))$ if $1\leq p<\gamma$, $\rho^{\eps}|\mathbf{u}^{\eps}|^2\in L^{\infty}(0,\infty;L^1(\Omega))$, $\rho^{\eps}\mathbf{u}^{\eps}\in C([0,\infty);w\mbox{-}L^{\frac{2\gamma}{\gamma+1}}(\Omega))$, and the following energy inequality holds for almost all $t>0$,
\begin{equation*}
  \int_{\Omega}\left\{\frac{1}{2}\rho^{\eps}|\mathbf{u}^{\eps}|^2+\pi^{\eps}\right\}(t)\,\mathrm{d}x+\int_0^t\int_{\Omega}\left\{\nu|\nabla_x \mathbf{u}^{\eps}|^2+(\xi+\nu)|\mathbf{u}^{\eps}|^2\right\}\,\mathrm{d}x\mathrm{d}s\leq E_0^{\eps}\,.
\end{equation*}
As analyzed in \cite{L1,LM}, if we assume $(\rho_0^{\eps})^{-1/2}\mathbf{m}_0^{\eps}$ converges weakly in $L^2(\Omega)^d$ to some $\mathbf{u}_0$ (since it is bounded in $L^2$ uniformly in $\eps$ from the point view of the initial bound \eqref{initialbound}), then in view of above energy estimates, we may assume that $\mathbf{u}^{\eps}$ convergences weakly in $L^2(0,\infty;H^1(\Omega))^d$ to some $\bar{\mathbf{u}}$, where $\bar{\mathbf{u}}$ solves the incompressible Navier-Stokes equations.
\begin{equation}\label{In-CNS}
 \begin{cases}
 \partial_t\bar{\mathbf{u}}+{\rm div}(\bar{\mathbf{u}}\otimes\bar{\mathbf{u}})-\nu\Delta\bar{\mathbf{u}}+\nabla\pi=0,\\
 {\rm div}\bar{\mathbf{u}}=0\quad \text{in}\quad \Omega\times(0,\infty),\quad \text{and}\quad \bar{\mathbf{u}}=0 \quad \text{on}\quad \partial\Omega\times(0,\infty),\\
 \bar{\mathbf{u}}|_{t=0}=\mathbf{u}_0,
 \end{cases}
 \end{equation}
in the sense of distributions for some $\pi$. Here we notice that in the second boundary condition of \eqref{eq:Robin BC}, as $\eps\rightarrow 0$, $(\mathbf{u}^\eps)^\tau \rightarrow 0$, thus in the limit $\bar{\mathbf{u}}=0$ on $\partial\Omega$.

By setting $\Psi^\eps = (\rho^\eps -1)/\eps$, $\pi^\eps = \frac{(\rho^\eps)^\gamma -1-\gamma(\rho^\eps-1)}{\eps^2 \gamma (\gamma-1)}$ and $\mathbf{m}^\eps=\rho^\eps\mathbf{u}^\eps$, the equations \eqref{eq:CNS} can be recast as
	\begin{align} \label{eq:CNS-v}
	  \begin{cases}
	 \partial_t\Psi^\eps+\left[ \frac{{\rm div}\mathbf{m}^\eps}{\eps}\right]=0,\\
	  \partial_t \mathbf{m}^\eps + \left[ \frac{\nabla_x \Psi^\eps}{\eps}\right]
	  	= \left[ {\rm div} \Sigma (\mathbf{u}^\eps) \right]
	  	  -{\rm div} (\mathbf{m}^\eps \otimes \mathbf{u}^\eps)
	  	  -(\gamma-1)\nabla_x \pi^\eps .
	  \end{cases}
	\end{align}
We mention that the bracket term in the left and right hand side of the equals in \eqref{eq:CNS-v} are the acoustic and diffusion terms respectively. Motivated by the formulation \eqref{eq:CNS-v}, as in \cite{DGLM99}, we denote $\phi=(\Psi,\, \mathbf{m})^\top$, and define the viscous wave operator $\mathcal{A}_\eps = \mathcal{A} + \eps \mathcal{D}$ where $\mathcal{A}$ and $\mathcal{D}$ are acoustic and diffusion operators respectively, defined on $\mathcal{D}'(\Omega)\times \mathcal{D}'(\Omega)^d$ by
\begin{equation}\nonumber
  \mathcal{A}\phi
= \begin{pmatrix}
    \mathrm{div}\mathbf{m} \\[2pt]
    \grad \Psi
  \end{pmatrix}\,,
\qquad
  \mathcal{D} \phi
= \begin{pmatrix} 0\\ {\rm div}\Sigma(\mathbf{m}) \end{pmatrix}
= \begin{pmatrix} 0\\
     \nu \Delta_x \mathbf{m} + (\xi+\nu)\grad \mathrm{div}\mathbf{m}
  \end{pmatrix}\,,
\end{equation}
and we assume that $\gamma>0$, $\nu > 0$ and $\xi + \nu > 0$.

Let $\{\lambda^2_{k,0}\}_{k\geq 1}$ be the nondecreasing sequence of eigenvalues and $\{\Psi_{k,0}\}_{k \geq 1}$ the orthonormal basis of $L^2(\Omega)$ functions with zero mean value of eigenvectors of the Laplace operator $-\Delta_x$ with homogeneous Neumann boundary conditions:
	\begin{equation}\label{Laplace}
		-\Delta_x \Psi_{k,0} = \lambda^2_{k,0} \Psi_{k,0} \quad \text{in } \Omega\,,
	\quad \frac{\partial \Psi_{k,0}}{\partial \mathrm{n}} = 0 \quad \text{on } {\partial\Omega}\,,
	\quad	\text{and \ } \int_\Omega \Psi_{k,0} = 0\,.
	\end{equation}
A solution of \eqref{Laplace} is said to be trivial if $\lambda_{k,0}=0$ and $\Psi_{k,0}$ is a constant. We will say $\Omega$ satisfies {\em assumption (H)} if all the solution of \eqref{Laplace} are trivial.
The eigenvalues and eigenvectors of $\A$ read as follows
	\begin{equation}\label{phi-int-0}
		\phi^{\pm}_{k,0} = \frac{1}{\sqrt{2}} \left( \Psi_{k,0}\,, \mathbf{m}^\pm_{k,0} = \pm \frac{\grad \Psi_{k,0}}{i \lambda_{k,0}} \right)^\top \in \mathbb{C}^{1+d}\,,
	\end{equation}
	\begin{equation}\label{acoustic-mode}
		\A \phi^{\pm}_{k,0} = \pm i \lambda_{k,0} \phi^{\pm}_{k,0} \quad \text{in } \Omega\,, \quad
		\mathbf{m}^\pm_{k,0} \cdot \mathbf{n} = 0 \quad \text{on } {\partial\Omega}\,.
	\end{equation}
In the paper \cite{DGLM99}, for the isentropic case, Desjardins, Grenier,  Lions
and  Masmoudi  constructed the viscous boundary layers under two conditions: the first, a geometric condition on the domain $\Omega$ (the ``assumption $(\mathrm{H})$" in \cite{DGLM99}); the second, an orthogonality condition for multiple eigenvalues, namely that if $\lambda_{k,0}=\lambda_{l,0}$ and $k\neq l$, then
\begin{equation}\label{condition}
  \int_{\partial\Omega} \grad \Psi_{k,0}\!\cdot\! \grad \Psi_{l,0}\mathrm{d}\sigma_{\!x} =0\,.
\end{equation}

Let us explain more as in \cite{JM15} how the condition \eqref{condition} should be enforced.
{\em Assume that $\lambda^2$ is an eigenvalue of \eqref{Laplace} and denote by $\mathrm{H}_0=\mathrm{H}_0(\lambda)$ the eigenspace associated to $\lambda^2$, i.e.
\begin{equation*}\label{Hspace}
     \mathrm{H}_0(\lambda)= \{\Psi \in \mathcal{D}(-\Delta_{\!x}): -\Delta_{\!x} \Psi= \lambda^2 \Psi\quad\!\! \mbox{in}\quad\!\! \Omega,\ \ \tfrac{\p \Psi}{\p \mathrm{n}}=0 \quad\!\! \mbox{on}\quad\!\! {\partial\Omega}\},
\end{equation*}
where $\mathcal{D}(-\Delta_{x})= H^2(\Omega) \cap \{\Psi : \tfrac{\p \Psi}{\p \mathrm{n}}=0 \text{ on } {\partial\Omega}\}$ denotes the domain of $-\Delta_x$ with Neumann boundary condition. On the finite dimensional space $\mathrm{H}_0(\lambda)$, we can define a quadratic
form $Q_1$   (its associated bilinear form is also  denoted by $Q_1$)
  and a symmetric operator $L_1=L^\lambda_1$ by
\begin{equation*}\label{L1def}
   Q_1(\Psi,\tilde{\Psi}) = \int_{\partial\Omega} \grad \Psi\!\cdot\! \grad \tilde{\Psi}\,\mathrm{d}\sigma_{\!x}= \int_\Omega L_1(\Psi)\tilde{\Psi}\,\mathrm{d}x\,.
\end{equation*}
Thus the condition \eqref{condition} can be restated as
\begin{equation*}\label{condition0}
     Q_1(\Psi_{k,0},\Psi_{l,0})=0\,,\quad \mbox{if}\quad\! \Psi_{k,0}, \Psi_{l,0}\in \mathrm{H}_0(\lambda)\quad\!\mbox{and}\quad\! k\neq l\,.
\end{equation*}
This condition means that the eigenvectors $\Psi_{k,0}$ for $\lambda_{k,0}=\lambda$ are orthogonal for the symmetric operator $L_1$. Of course, since $L^2(\Omega)$ is the direct sum of the spaces $\mathrm{H}_0(\lambda)$ for different $\lambda's$,
from the definition of $L_1$ on each eigenspace $\mathrm{H}_0(\lambda)$, we can define an operator $L^\lambda_1$ on $L^2(\Omega)$ which leaves each eigenspace $\mathrm{H}_0(\lambda)$ invariant. But this is not necessary, so we will think of $L^\lambda_1$ as acting on $\mathrm{H}_0(\lambda)$ for a fixed multiple eigenvalue $\lambda$.

The orthogonality condition \eqref{condition} turns out to be enough for the construction of the boundary layer if the eigenvalues of $L_1$ are simple, namely, if
\begin{equation*}\label{condition1}
    \lambda_{k,1}= \int_{\partial\Omega}|\grad \Psi_{k,0}|^2\,\mathrm{d}\sigma \neq \int_{\partial\Omega}|\grad \Psi_{l,0}|^2\,\mathrm{d}\sigma_{\!x} = \lambda_{l,1}
\end{equation*}
for all $l\neq k$ such that $\lambda_{l,0}=\lambda_{k,0}=\lambda$. However, if $\lambda_1$ is an eigenvalue of $L_1$ with multiplicity greater than or equal to $2$, more precisely, if there exists $l\neq k$ such that $\lambda_{l,0}=\lambda_{k,0}=\lambda$, and $\lambda_{l,1}=\lambda_{k,1}=\lambda_1$, then we need an extra orthogonality condition as following: Let $\mathrm{H}_1=\mathrm{H}_1(\lambda_1)$ be defined by
\begin{equation*}\label{Hspace1}
    \mathrm{H}_1(\lambda_1)= \{ \Psi\in \mathrm{H}_0(\lambda) : L_1 \Psi = \lambda_1 \Psi \}\,.
\end{equation*}
 On the finite dimensional space $\mathrm{H}_1$, there exists a quadratic form $Q_2$ and a symmetric operator $L_2$ [see the definition \eqref{L2} below], the extra condition is
\begin{equation*}\label{condition11}
    Q_2(\Psi_{k,0},\Psi_{l,0})=0\,,\quad \mbox{if}\quad\! \Psi_{k,0},
    \Psi_{l,0}\in \mathrm{H}_1(\lambda_1)\quad\!\mbox{and}\quad\! k\neq l\,.
\end{equation*}
This condition is enough if $L_2$ has only simple eigenvalue on the vector space $\mathrm{H}_1$. However, if $L_2$ has eigenvalue with multiplicity greater than or equal to 2, we need additional condition on $\mathrm{H}_2$, and so on. This process can be continued inductively to find the  condition we have to impose on the eigenvectors of $-\Delta_x$.}

The present paper is strucured as follows. In section 2, we will focus on explanation of Mach-number-depending Navier-slip boundary condition. In section 3, we will constuct the boundary layer and prove the main Proposition \ref{thm:boundary-layer}. In section 4, we will give the details of proof of strong convergence of $\mathcal{Q} \mathbf{u}^\eps$. Finally, in the appendix, we collect some lemmas which will be used in the proof of the Proposition \ref{thm:boundary-layer}. We end this short introduction with the statement of our result of this paper, which is the same as the Proposition 2 in \cite{DGLM99}, except that we add the higher order orthogonality conditions. As we mentioned above, the main purpose of the present paper is to fill the gap in the proof of the Proposition 2 in \cite{DGLM99} and to rewrite it in the case of Navier-slip boundary condition.
 \begin{theorem} \label{thm:Incompressible Limit}
Under the above assumption conditions, when the isentropic compressible Navier-Stokes equations \eqref{eq:CNS} equipped with the Navier-Slip boundary conditions \eqref{eq:Robin BC}, then $\rho^{\eps}$ converges to 1 in $C([0,T];L^{\gamma}(\Omega))$ and $\mathbf{u}^{\eps}$converges to $\bar{\mathbf{u}}$ weakly in $L^2((0, T)\times\Omega)^d$ for all $T>0$ and strongly if $\Omega$ satisfies (H). In addition, $\bar{\mathbf{u}}$ is a global weak solution of the incompressible Navier-Stokes equations \eqref{In-CNS} satisfying $\bar{\mathbf{u}}|_{t=0}=\mathbf{u}_0$ in $\Omega$.
\end{theorem}
\begin{remark}
 In fact, in this theorem, $\rho^{\eps}$ converges to 1 in $C([0,T];L^{\gamma}(\Omega))$ and $\mathbf{u}^{\eps}$converges to $\bar{\mathbf{u}}$ weakly in $L^2((0, T )\times\Omega)^d$ are classical results which have been done in a great quantity literature, for example, \cite{LM,DGLM99}. In the following, we will only construct the acoustic boundary layers which yields the strong convergences if $\Omega$ satisfies (H).
\end{remark}


\section{Boundary Conditions depending upon the Mach Number} 
\label{sec:BC_upon_Mach}


\subsection{Navier-slip boundary conditions of Navier-Stokes equation and Maxwell boundary conditions of Boltzmann equation} 
To understand the Navier-slip boundary conditions given before which depends on the Mach number, it may be a better way to start from a viewpoint of the hydrodynamic limit at the microscopic (kinetic) level. We introduce the initial-boundary value problem of the rescaled Boltzmann equation in a bounded domain $\Omega$,
\begin{equation} \label{eq:BE}
  \begin{cases}
	\varepsilon \partial_t F_\varepsilon (t,\, x,\, v) + v \cdot \nabla_{x} F_\varepsilon (t,\, x,\, v) = \frac{1}{\varepsilon} \mathcal{Q}(F_\varepsilon, F_\varepsilon),\\
F_\varepsilon (t,\, x,\, v)|_{t=0}=F_\varepsilon^{in}(x,\, v) \ge 0,\\
F_\varepsilon (t,\, x,\, v)|_{\mathit{\Sigma}_-} = (1-\alpha_\varepsilon) \mathcal{R} F_\varepsilon (t,\, x,\, v)|_{\mathit{\Sigma}_+}+\alpha_\varepsilon \mathcal{K} F_\varepsilon (t,\, x,\, v)|_{\mathit{\Sigma}_+},
  \end{cases}
\end{equation}
where the unknown function $F = F(t,\, x,\, v)$ denotes the density distribution function of molecules located at position $x \in \Omega$ with velocity $v \in \mathbb{R}^3$ at time $t \ge 0$. The right-hand side of the first equation is the Boltzmann collision operator $\mathcal{Q}(F, F)$. We should mention that we will not go into particulars since our main concern in this subsection is the reasonableness of the kind of Navier-slip boundary conditions depending upon the Mach number.

The boundary condition \eqref{eq:BE}$_3$ is the so-called Maxwell's condition which describes the balance between the outgoing and incoming part of the trace of $F$, where we defined the outgoing and incoming sets $\mathit{\Sigma}_+$ and $\mathit{\Sigma}_-$ on the boundary $\partial \Omega$ by
\begin{equation*}
	 \mathit{\Sigma}_\pm
	 =\left\{ (x,\, v) \in \partial \Omega \times \mathbb{R}^3, \ \pm v \cdot n(x) >0 \right\}
\end{equation*}
with respect to the outward normal $n(x)$ at point $x \in \partial \Omega$. In addition, the notations $\mathcal{R}$ and $\mathcal{K}$ stand for the specular reflection operator and diffusion reflection operator, respectively, whose exact definition can be found in many references and we will not expand in this point. At last, the non-dimensional parameter $\alpha_\eps$ in equation \eqref{eq:BE}$_3$ is called the accommodation coefficient, which measures the roughness of the boundary. In a general setting, it takes the form of $\alpha_\eps = \sqrt{2\pi} \chi \eps^\beta$ with $\chi \ge 0,\, \beta \in [0,\, 1]$.



It should be pointed out that the non-dimensional parameter $\eps >0$ is the Knudsen number defined as the ratio of the mean free path and the macroscopic length scale. Hydrodynamic regimes are those in which the Knudsen number $\eps$ is small.

Formally, under the Navier-Stokes scaling where the Mach number is of the same order as the Knudsen number, the corresponding hydrodynamic equations can be derived, which takes the form of the incompressible Navier-Stokes-Fourier equations, and the compressible ``Navier-Stokes-Fourier asmptotes'' in an asymptotic sense,
\begin{align} \nonumber
	\begin{cases}
   \partial_t \rho^\eps + \frac{1}{\eps} {\rm div} \mathbf{u}^\eps = 0 ,\\ 
   \partial_t \mathbf{u}^\eps + \frac{1}{\eps} \nabla_x (\rho^\eps + \theta^\eps) + {\rm div}\left(\mathbf{u}^\eps\otimes\mathbf{u}^\eps \right)=\nu \left( \nabla_x \mathbf{u}^\eps + \nabla_x^\top \mathbf{u}^\eps-\frac{2}{3} {\rm div} \mathbf{u}^\eps Id \right),\\
  \frac{3}{2} \partial_t \theta^\eps + \frac{1}{\eps} {\rm div} \mathbf{u}^\eps+\mathbf{u}^\eps \cdot\nabla_x\theta^\eps=\kappa\Delta_x \theta^\eps.
	\end{cases}
\end{align}
Correspondingly, the fluid boundary conditions can also be derived, i.e.,
\begin{align} \label{eq: NSF-Slip BC}
	\begin{cases}
	\mathbf{u}^\eps \cdot n = 0, \text{\quad on } \partial \Omega,\\
	\nu \left[ \sigma(\mathbf{u}^\eps) \cdot n\right] ^\tau +\chi\eps^{\beta-1} 
	(\mathbf{u}^\eps)^\tau =0, \text{\quad on } \partial \Omega,\\
	\kappa \nabla_x \theta^\eps \cdot n + \frac{4}{5}\chi \eps^{\beta-1}\theta^\eps =0,
	\text{\quad on } \partial \Omega.
	\end{cases}
\end{align}
Recalling that we only deal with the isentropic case due to the difficulties of the Lions type renormalized weak solutions to Navier-Stokes-Fourier system, the equations \eqref{eq: NSF-Slip BC} will reduce to the preceding Navier-slip boundary conditions \eqref{eq:Robin BC}, in this paper, we only want to reveal the same ``damping effect" aroused by the Navier-slip boundary condition as the Dirichlet boundary condition\cite{DGLM99}, so we take a particular case $\beta = \frac{1}{2}$ for simplicity.

\subsection{Dual formation and dual boundary conditions} \label{sub:dual_formation}
To obtain the dual formation of the compressible Navier-Stokes system \eqref{eq:CNS-v}, we take inner product of the momentum equation with the test function $\mathbf{m}$ satisfying $\mathbf{m} \cdot n =0$, and get that
\begin{align}\nonumber
  \frac{\rm d}{{\rm d} t}\skp{\mathbf{m}^\eps}{\mathbf{m}}+\frac{1}{\eps} \skp{\nabla_x \Psi^\eps}  {\mathbf{m}}-\skp{{\rm div} \Sigma(\mathbf{u}^\eps)}{\mathbf{m}}=\skp{-{\rm div}(\mathbf{m}^\eps  \otimes \mathbf{u}^\eps)-(\gamma-1) \nabla_x \pi^\eps}{\mathbf{m}}.
\end{align}
Define the ``dual'' Navier-slip boundary conditions for $\mathbf{m}$ as
\begin{align}\label{Robin-BC-dual}
  \begin{cases}
	 \mathbf{m}\cdot n =0, &\text{ on } \partial \Omega \times (0,\infty), \\
	  \nu [\sigma (\mathbf{m}) \cdot n]^\tau+\chi\sqrt\eps^{-1} \mathbf{m}^\tau =0, &\text{ on }\partial\Omega \times (0,\infty),
  \end{cases}
\end{align}
then we can check the following ``dual'' property for the stress tersor
\begin{align}\nonumber
  \skp{{\rm div}\Sigma(\mathbf{u}^\eps)}{\mathbf{m}}=\skp{\mathbf{u}^\eps}{{\rm div}\Sigma(\mathbf{m})}.
  \end{align}
Indeed, by taking advantage of integrating by parts, we get
\begin{align*}
	 \skp{{\rm div} \Sigma(\mathbf{u}^\eps)}{\mathbf{m}}
	 = &-\skp{\Sigma(\mathbf{u}^\eps)}{\nabla_x \mathbf{m}}+\int_{\partial \Omega}\nu[\sigma (\mathbf{u}^\eps) \cdot n]^\tau \mathbf{m}^\tau \d \sigma_x\\
	 = &- \skp{\nabla_x \mathbf{u}^\eps}{\Sigma(\mathbf{m})} - \int_{\partial \Omega} \chi \sqrt\eps^{-1} (\mathbf{u}^\eps)^\tau \mathbf{m}^\tau \d \sigma_x\\
	 = &\skp{\mathbf{u}^\eps}{{\rm div} \Sigma(\mathbf{m})} - \int_{\partial \Omega} (\mathbf{u}^\eps)^\tau \left\{ \nu [\sigma (\mathbf{m}) \cdot n]^\tau + \chi \sqrt\eps^{-1} \mathbf{m}^\tau \right\} \d \sigma_x\\
	 = &\skp{\mathbf{u}^\eps}{{\rm div} \Sigma(\mathbf{m})}.
\end{align*}
On the other hand, the dual boundary condition $\mathbf{m} \cdot n =0$ also yields that
\begin{equation*}
  \skp{\nabla_x \Psi^\eps}{\mathbf{m}}=-\skp{\Psi^\eps}{{\rm div} \mathbf{m}},\quad
  \skp{-\nabla_x \pi^\eps}{\mathbf{m}}=\skp{\pi^\eps}{{\rm div} \mathbf{m}},
\end{equation*}
and likewise, the original boundary condition $\mathbf{u}^\eps \cdot n =0$ yields that
\begin{align*}
	\skp{-{\rm div} (\mathbf{m}^\eps \otimes \mathbf{u}^\eps)}{\mathbf{m}}=\skp{\mathbf{m}^\eps \otimes \mathbf{u}^\eps}{\nabla_x \mathbf{m}}.
\end{align*}
Thus, combined with the inner product of the master equation with test function $\Psi$, the above discuss implies the ``dual'' formation of the compressible Navier-Stokes system \eqref{eq:CNS-v}, i.e.,
\begin{align} \label{eq:CNS dual}
  \begin{cases}
  \frac{\rm d}{{\rm d} t} \skp{\Psi^\eps}{\Psi}-\frac{1}{\eps}\skp{\mathbf{m}^\eps}{\nabla_x \Psi}=0,\\
	  \begin{aligned}
	  \frac{\rm d}{{\rm d} t} \skp{\mathbf{m}^\eps}{\mathbf{m}}
	   &-\frac{1}{\eps} \skp{\Psi^\eps}{{\rm div} \mathbf{m}}
	  -\skp{\mathbf{m}^\eps - \eps \Psi^\eps \mathbf{u}^\eps}{{\rm div} \Sigma (\mathbf{m})}\\
	   =&-\eps \skp{\Psi^\eps \mathbf{u}^\eps}{{\rm div} \Sigma (\mathbf{m})}+\skp{\mathbf{m}^\eps \otimes\mathbf{u}^\eps}{\nabla_x \mathbf{m}}+(\gamma-1)\skp{\pi^\eps}{{\rm div} \mathbf{m}},
	  \end{aligned}
  \end{cases}
\end{align}
where we have used the fact $\mathbf{u}^\eps = \mathbf{m}^\eps - \eps \Psi^\eps \mathbf{u}^\eps$.

Let $\phi^\eps = (\Psi^\eps,\, \mathbf{m}^\eps)^\top$ and $\phi = (\Psi,\, \mathbf{m})^\top$, then we can write
	\begin{align}\label{eq:short dual form}
	  \frac{\rm d}{{\rm d} t} \skp{\phi^\eps}{\phi} - \frac{1}{\eps} \skp{\phi^\eps}{\mathcal{A}_\eps \phi} = c^\eps(\mathbf{m}),
	\end{align}
here we used the notations $\mathcal{A}_\eps = \mathcal{A} + \eps \mathcal{D}$ and
	\begin{align*}
	  c^\eps(\mathbf{m})
	= - \eps \skp{\Psi^\eps \mathbf{u}^\eps}{{\rm div} \Sigma (\mathbf{m})}
	 	+ \skp{\mathbf{m}^\eps \otimes \mathbf{u}^\eps}{\nabla_x \mathbf{m}}
	  + (\gamma-1) \skp{\pi^\eps}{{\rm div} \mathbf{m}} .
	\end{align*}

\begin{remark}
\begin{enumerate}
\item The dual boundary condition proposed here should be understood in the sense of dual equations \eqref{eq:CNS dual}, rather than in the sense of dual diffusion operator, which relies on the fact that the equation \eqref{eq:CNS-v}$_2$ concerns the momentum $\mathbf{m}^\eps = \rho^\eps \mathbf{u}^\eps$ while the diffusion contribution ${\rm div} \Sigma(\mathbf{u}^\eps)$ contains only the velocity $\mathbf{u}^\eps$. The advantage of the definition \eqref{Robin-BC-dual} is the boundary integration will vanishes as appeared in the equation \eqref{eq:short dual form}.
\item On the other hand, we mention here that by noticing that $\mathbf{u}^{\eps}=\mathbf{m}^{\eps}-\eps\Psi^{\eps}\mathbf{u}^{\eps}$, then $\eqref{eq:Robin BC}_2$ can be recast as
\begin{equation*}
    \nu [\sigma (\mathbf{m}) \cdot n]^\tau + \chi \sqrt\eps^{-1} \mathbf{m}^\tau
    -\eps\left\{\nu [\sigma (\Psi^{\eps}\mathbf{u}^{\eps})\cdot n]^\tau
    +\chi\eps^{-1/2} [{\Psi^{\eps}\mathbf{u}^{\eps}}]^\tau\right\}=0,
\end{equation*}
which is equivalent to $\eqref{Robin-BC-dual}_2$ in the sense of ignoring the higher order term of $\eps$ in above expansion equation.
\item In the next section for analyzing the boundary layer, our starting point is the dual operator $\mathcal{A}_\eps = \mathcal{A} + \eps \mathcal{D}$ with respect to the original operator $\mathcal{A} - \eps \mathcal{D}$ in the sense of dual formation \eqref{eq:CNS dual}. That is why we are required to consider the above dual boundary conditions.
\end{enumerate}
\end{remark}

\section{Construction of boundary layer} 

\begin{proposition} \label{thm:boundary-layer}
Let $\Omega$ be a $C^2$ bounded domain of $\mathbb{R}^d$ and let $k\geq 1$, $N\geq 0$. Let the eigenvectors $\Psi_{k,0}$ of $-\Delta_x$ satisfy the orthogonality conditions \eqref{condition} and \eqref{orth-cond}. Then, there exists approximate eigenvalues $i\lambda^\pm_{k,\eps,N}$ and eigenvectors $\phi^\pm_{k,\eps,N}=(\Psi^\pm_{k,\eps,N}, \mathbf{m}^\pm_{k,\eps,N})^\top $ of $\mathcal{A}_\eps$ such that
\begin{equation*}\label{spectral}
   \mathcal{A}_\eps \phi^\pm_{k,\eps,N} = i\lambda^\pm_{k,\eps,N} \phi^\pm_{k,\eps,N} + \mathcal{R}^\pm_{k,\eps,N}\,,
\end{equation*}
with
\begin{equation*}\label{spectral2}
   i \lambda^\pm_{k,\eps,N}= \pm i \lambda_{k,0} + i \lambda^\pm_{k,1}\sqrt{\eps} + O(\eps)\,,
\end{equation*}
and the real part of $i \lambda^\pm_{k,1}$, i.e. $\mathfrak{Re}(i \lambda^\pm_{k,1})\leq0$. Furthermore, for all $p\in [1,\infty]$, we have
\begin{equation}\label{error}
    \|\mathcal{R}^\pm_{k,\eps,N}\|_{L^p(\Omega)} \leq C_p (\sqrt{\eps})^{N+\frac{1}{p}}
    \quad \text{ and \ }
    \|\phi^\pm_{k,\eps,N} - \phi^\pm_{k,0} \|_{L^p(\Omega)} \leq C_p (\sqrt{\eps})^{\frac{1}{p}} \,.
	\end{equation}
\end{proposition}
\begin{remark} 
 \begin{itemize}
    \item[(1)]In our isentropic case, to get that $\mathfrak{Re}(i \lambda^\pm_{k,1})<0$, we need  an additional geometric condition on the domain $\Omega$, the ``assumption (H)''  in \cite{DGLM99}. For the non-isentropic case \cite{JM15}, this geometric condition is not needed because of the additional dissipative effect coming from heat conductivity, see the remark texts after \eqref{strict-disp}.
    \item[(2)]The proof gives a more precise expansion  of the eigenvalues $i\lambda^\pm_{k,\eps,N}$ and eigenvectors $\phi^\pm_{k,\eps,N}$.
\end{itemize}
\end{remark}
\begin{proof}
The main idea is to construct a boundary layer similar to the one  in \cite{DGLM99}. For the convenience of the readers, we give a complete proof here. The main idea is to build approximate modes of $\mathcal{A}_\eps$ in terms of $\phi^{\pm}_{k,0}$, we formally solve the equation
\begin{equation} \label{eigen-equation}
	\mathcal{A}_\eps \phi^\pm_{k,\eps,N}
	 =i\lambda^\pm_{k,\eps,N} \phi^\pm_{k,\eps,N} + \mathcal{R}^\pm_{k,\eps,N}\,,
\end{equation}
where we make for $\lambda^\pm_{k,\eps,N}$ and $\phi^\pm_{k,\eps,N}$ the following ansatz,
\begin{align}\label{ansatz-1}
 &\lambda^\pm_{k,\eps,N}= \sum^N_{i=0} \sqrt{\eps}^i \lambda ^\pm_{k,i},\ \text{ and }\\ \label{ansatz}
  & \phi^\pm_{k,\eps,N}
	   =\sum_{i=0}^N \sqrt{\eps}^i \left(\phi_{k,i}^{\pm,\Int}(x)
	   +\phi_{k,i}^{\pm,\b} \left( \pi(x),\, \zeta(x) \right) \chi(x) \right), 
	   \text{ with } \zeta(x) = \frac{\mathrm{d}(x)}{\sqrt{\eps}},
\end{align}		
where $\phi^{\pm,\Int}_{k,i}= (\Psi^{\pm,\Int}_{k,i}\,, \mathbf{m}^{\pm,\Int}_{k,i})^\top $ and $\phi^{\pm,\b}_{k,i}= (\Psi^{\pm,\b}_{k,i}\,, \mathbf{m}^{\pm,\b}_{k,i})^\top $ are smooth functions with the Navier-slip boundary conditions
\begin{align}\tag{BC$_i$} \label{eq:BC$_i$}
	\begin{cases} \displaystyle
	 \sum_{i=0}^N\sqrt{\varepsilon}^i\left(\mathbf{m}_{k,\, i}^{\pm,\Int}
	 +\mathbf{m}_{k,\,i}^{\pm,\b} \right) \cdot n = 0,\\
	  	\begin{aligned}
   \sum_{i=0}^N \sqrt{\varepsilon}^i \Bigg\{\nu\left[\left(\nabla_x\mathbf{m}_{k,\, i}^{\pm,\Int}
   +\nabla_x^\top \mathbf{m}_{k,\, i}^{\pm,\Int} \right) \cdot n \right] \cdot \nabla_x \pi^\alpha
   +\nu|\nabla_x \pi^\alpha|^2 \partial_{\pi^\alpha} (\mathbf{m}_{k,\, i}^{\pm,\b} \cdot n)& \\
   -\frac{\nu}{\sqrt{\varepsilon}}\partial_{\zeta}(\mathbf{m}_{k,i}^{\pm,\b}\cdot \nabla\pi^{\alpha})+\chi\varepsilon^{\frac{1}{2}-1}(\mathbf{m}_{k,i}^{\pm,\Int}+\mathbf{m}_{k,i}^{\pm,\b})\cdot\nabla\pi^{\alpha} \Bigg\} &=0,
       \end{aligned}
	 \end{cases}
\end{align}
on ${\partial\Omega}$, $\phi^{\pm,\b}_{k,i}$ being rapidly decreasing to $0$ in the $\zeta$ variable. We also require that
	\begin{equation*}\label{ortho2}
	   \< \phi^{\pm,\Int}_{k,j} \ |\ \phi^{\pm,\Int}_{k,0}\> = 0
	\end{equation*}
for all $j \geq 1$. Here the function $\mathrm{d}(x)$ is defined in \eqref{function-d}.

In the ansatz \eqref{ansatz}, $\chi(y)\in C^\infty_0(\Omega)$ is a smooth cut-off function such that $\chi(y)=1$ in a neighborhood of ${\partial\Omega}$ and $\chi(y)= 0$ if $\mathrm{d}(y) > \delta$ for some $\delta$ small enough. Here $\delta>0$ is to be chosen so that $\pi$ is uniquely determined in $\{0 < \mathrm{d}(y) < \delta \}$, as indicated by Lemma \ref{Projection}.

Straightforward calculations show that for $\phi^\b=(\Psi^\b\,, \mathbf{m}^\b)^\top $,
\begin{equation*}
		\mathcal{A}_\eps \phi^\b
	= \tfrac{1}{\sqrt{\eps}}\mathcal{A}^\mathrm{d} \phi^\b
		+ \mathcal{A}^\pi \phi^\b + \mathcal{D}^\b \phi^\b
		+ \sqrt{\eps} \mathcal{F}^\b \phi^\b + \eps \mathcal{G}^\b \phi^\b\,,
\end{equation*}
where
\begin{equation}\nonumber
		\mathcal{A}^\mathrm{d} \phi^\b = \begin{pmatrix}
		\p_\zeta(\mathbf{m}^\b\!\cdot\!\grad \mathrm{d})\\ 
		\p_\zeta\Psi^\b\grad \mathrm{d}\end{pmatrix}\,,\quad
		\mathcal{A}^\pi \phi^\b = \begin{pmatrix}
		\p_{\pi^\alpha}(\mathbf{m}^\b\!\cdot\! \grad \pi^\alpha)\\ 
		\p_{\pi^\alpha}\Psi^\b\grad \pi^\alpha
		   \end{pmatrix}\,,
	\end{equation}
\begin{equation}\nonumber
		\mathcal{D}^\b \phi^\b = \begin{pmatrix} 0 \\ \nu |\grad \mathrm{d}|^2 \partial^2_{\!\zeta}\mathbf{m}^\b + (\xi +\nu)\partial^2_{\!\zeta}(\mathbf{m}^\b\!\cdot\!
		\grad\mathrm{d} )\grad\mathrm{d}\end{pmatrix}\,,
\end{equation}
\begin{equation}\nonumber
		\begin{aligned}
		  \mathcal{F}^\b \phi^\b
		 = & \begin{pmatrix}
		 	0\\
	2 \nu \partial^2_{\!\zeta \pi^\alpha}\mathbf{m}^\b(\grad\mathrm{d}\!\cdot\! \grad\pi^\alpha) + (\xi+ \nu)[(\p^2_{\!\zeta \pi^\alpha}\mathbf{m}^\b\!\cdot\!\grad\pi^\alpha)\grad\mathrm{d} + (\p^2_{\!\zeta \pi^\alpha}\mathbf{m}^\b\!\cdot\!\grad\mathrm{d})\grad\pi^\alpha]
		 		 \end{pmatrix}\\ &
				+\begin{pmatrix}
			0\\
\nu \partial_{\zeta}\mathbf{m}^\b\Delta_{\!x}\mathrm{d} + (\xi + \nu) \partial_{\zeta}\mathbf{m}^\b\!\cdot\! \grad^2\mathrm{d}
					\end{pmatrix}\,,
		\end{aligned}
\end{equation}
and
\begin{equation}\nonumber
	\begin{aligned}
		  \mathcal{G}^\b \phi^\b
		= & \begin{pmatrix} 
		0\\
	\nu \left[ \p_{\pi^\alpha} \mathbf{m}^\b \Delta_{\!x}\pi^\alpha+\p^2_{\!\pi^\alpha\pi^\beta}\mathbf{m}^\b(\grad\pi^\alpha\!\cdot\!\grad\pi^\beta) \right]
		\end{pmatrix}\\
		  &+\begin{pmatrix} 
		 0\\ 
		 (\xi + \nu) \left[ \p_{\!\pi^\alpha}\mathbf{m}^\b\!\cdot\! \grad^2\pi^\alpha 
		 +(\p^2_{\!\pi^\alpha\pi^\beta}\mathbf{m}^\b\!\cdot\!\grad\pi^\alpha)\grad\pi^\beta \right]
		  	 \end{pmatrix}\,.
	\end{aligned}
\end{equation}
We will use the following notations for the simplicity: $\mathcal{F}^\b \phi^\b = (0, \mathcal{F}^\mathbf{m})^\top $ and $\mathcal{G}^\b \phi^\b = (0, \mathcal{G}^\mathbf{m})^\top $, and furthermore,
\begin{equation}\nonumber
	\begin{aligned}
	  \F^\mathbf{m}\!\cdot \grad\pi &= \F^\mathbf{m}_{\pi\pi}(\mathbf{m}^\b\!\cdot\!\grad\pi) +\F^\mathbf{m}_{\pi d}(\mathbf{m}^\b\!\cdot\!\grad \mathrm{d})\,,\\
	  \F^\mathbf{m}\!\cdot \grad\mathrm{d} &=\F^\mathbf{m}_{d\pi}(\mathbf{m}^\b\!\cdot\!\grad\pi) + \F^\mathbf{m}_{d d}(\mathbf{m}^\b\!\cdot\!\grad \mathrm{d})\,,\\
	 \G^\mathbf{m}\!\cdot \grad\pi &=\G^\mathbf{m}_{\pi\pi}(\mathbf{m}^\b\!\cdot\!\grad\pi) +\G^\mathbf{m}_{\pi d}(\mathbf{m}^\b\!\cdot\!\grad \mathrm{d})\,,\\
	\G^\mathbf{m}\!\cdot \grad\mathrm{d} &=\G^\mathbf{m}_{d\pi}(\mathbf{m}^\b\!\cdot\!\grad\pi) + \G^\mathbf{m}_{d d}(\mathbf{m}^\b\!\cdot\!\grad \mathrm{d})\,,
	\end{aligned}
\end{equation}
where $\F^\mathbf{m}_{\pi\pi},\ \F^\mathbf{m}_{\pi d},\ \F^\mathbf{m}_{d\pi},\ \F^\mathbf{m}_{d d},\ \G^\mathbf{m}_{\pi\pi},\ \G^\mathbf{m}_{\pi d},\ \G^\mathbf{m}_{d\pi},\ \G^\mathbf{m}_{d d}$ are linear functions. By substituting \eqref{ansatz-1} and \eqref{ansatz} into \eqref{eigen-equation},  now we start to solve the equation \eqref{eigen-equation} inductively.

{\bf \underline{Step 0 :}} First, the order $\sqrt{\eps}^{-1}$ in the boundary layer gives $\mathcal{A}^\mathrm{d} \phi^\b_{k,0}=0$ which implies that
\begin{equation}\nonumber
\mathbf{m}^\b_{k,0}\!\cdot\! \grad \mathrm{d}=0\quad \mbox{and}\quad\! \Psi^\b_{k,0} =0\,.
\end{equation}
In particular, $\mathbf{m}^\b_{k,0}\!\cdot\!\mathrm{n}=0$ and consequently $\mathbf{m}^\Int_{k,0}\!\cdot\!\mathrm{n}=0$ on the boundary ${\partial\Omega}$.

The order $\sqrt{\eps}^0$ of the interior part in \eqref{eigen-equation} implies that
\begin{equation}\nonumber
\mathcal{A} \phi^{\pm,\Int}_{k,0}= i \lambda^\pm_{k,0} \phi^{\pm,\Int}_{k,0}\,.
\end{equation}
Comparing with \eqref{phi-int-0} and \eqref{acoustic-mode}, we have $\phi^{\pm,\Int}_{k,0}=\phi^{\pm}_{k,0}$ and $\lambda^\pm_{k,0} = \pm \lambda_{k,0}$.

{\bf Remark:} From now on, for the simplicity of representation, we only calculate for $``+"$ case, the calculation for the $``-"$ case is the same.
\subsection{Order $\sqrt{\eps}^0$ of the boundary layer}
The order $\sqrt{\eps}^0$ in the boundary layer gives
	\begin{equation*} \label{eq:b-0-fold}
	  -\mathcal{A}^\mathrm{d} \phi^\b_{k,1}
	= (\A^\pi+\mathcal{D}^\b-i \lambda_{k,0})\phi^\b_{k,0},
	\end{equation*}
which equally means that,
\begin{multline} \label{eq:b-0-unfold}
		- \begin{pmatrix}
			\p_\zeta(\mathbf{m}^\b_{k,1} \cdot \grad \mathrm{d})\\
			 \p_\zeta\Psi^\b_{k,1} \grad \mathrm{d}
				\end{pmatrix}
		=\begin{pmatrix}
			\p_{\pi^\alpha}(\mathbf{m}^\b_{k,0}\!\cdot\! \grad \pi^\alpha)\\
			\p_{\pi^\alpha}\Psi^\b_{k,0} \grad \pi^\alpha
		\end{pmatrix}\\
	+\begin{pmatrix}
		0 \\
	\nu|\grad \mathrm{d}|^2 \partial^2_{\!\zeta\!\zeta}\mathbf{m}^\b_{k,0} + (\xi +\nu)\partial^2_{\!\zeta\!\zeta}(\mathbf{m}^\b_{k,0}\!\cdot\! \grad\mathrm{d} )\grad \mathrm{d}
	\end{pmatrix}
  -i \lambda_{k,0} \begin{pmatrix} \Psi^\b_{k,0} \\ \mathbf{m}^\b_{k,0} \end{pmatrix}\,.
\end{multline}

The process to solve this system will be illustrated in the following steps, which are the foundation to solve the more complicated ODE for  $\phi^\b_{k,j}$. The process is smilar as \cite{JM15} but with different boundary condition, for the consideration of completeness, we still write all the details. Before we solve the system \eqref{eq:b-0-unfold}, we recall that we  already know the values of $\mathbf{m}^\b_{k,0}\!\cdot\!\grad \mathrm{d}$ and $\Psi^\b_{k,0}$ as well as the boundary data for
 $\mathbf{m}^\b_{k,0}\!\cdot\!\grad \mathrm{\pi}$ and $\Psi^\b_{k,0}$.

{\bf \underline{Step 1}} is  to solve  $\Psi^\b_{k,1}$. This is achieved by taking the inner product of the second equation of \eqref{eq:b-0-unfold}
with $\grad\mathrm{d}$. This gives $\p_{\zeta}\Psi^\b_{k,1}=0$ thus $\Psi^\b_{k,1}=0$, noting that $\grad \pi^\alpha\!\cdot\!\grad \mathrm{d}=0$ and $\mathbf{m}^\b_{k,0}\!\cdot\! \grad\mathrm{d}=0$.

{\bf \underline{Step 2}} is to  solve the tangential part of $\mathbf{m}^\b_{k,0}$, i.e. $\mathbf{m}^\b_{k,0}\!\cdot\!\grad\pi$. By taking the inner product of the second equation of \eqref{eq:b-0-unfold} with $\grad \pi$, we get that $\mathbf{m}_{k,0}^\b \cdot \grad\pi$ satisfies the ODE
\begin{equation}\label{ODE-mb-0}
	\begin{cases}
	\mathcal{L} (\mathbf{m}_{k,0}^\b \cdot \nabla_x \pi) = \left( \nu |\nabla_x \mathrm{d}|^2 \partial_{\zeta \zeta}^2-i\lambda_{0} \right) (\mathbf{m}_{k,0}^\b \cdot \nabla_x \pi)=0,\\
	(\nu \partial_\zeta-\chi) (\mathbf{m}_{k,0}^\b \cdot \nabla_x \pi) |_{\zeta =0} = \chi \mathbf{m}_{k,0}^{\Int}\cdot \nabla_x \pi, \quad \text{on } \partial \Omega,\\
	(\mathbf{m}_{k,0}^\b \cdot \nabla_x \pi)|_{\zeta \to + \infty} \to 0.
	\end{cases}
\end{equation}		
The boundary condition in second line of \eqref{ODE-mb-0} should be understood as: when $\zeta = \frac{\mathrm{d}(x)}{\sqrt{\eps}} =0$, $x\in {\partial\Omega}$. The solution of \eqref{ODE-mb-0} is given by
\begin{equation} \label{sol-b-u-0}
	\begin{aligned}
       \mathbf{m}^\b_{k,0}(\pi(x),\zeta)\!\cdot\!\grad\pi
	\triangleq \widetilde{Z}^{\b,\mathrm{u}}_0(\zeta,\phi^\Int_{k,0})
	=\frac{\chi}{\chi+\nu c_\nu} (\mathbf{m}^\Int_{k,0} \cdot \grad \pi) e^{-c_\nu \zeta}\,,
	\end{aligned}
\end{equation}
where $c_\nu=\frac{1+i}{2} \sqrt{\frac{2\lambda_{k,0}}{\nu}} \frac{1}{|\grad\mathrm{d}|}$.
Note that  on the right-hand side of \eqref{sol-b-u-0}, $\mathbf{m}^\Int_{k,0}\!\cdot\!\grad\pi$ is  taken on the boundary ${\partial\Omega}$, i.e. $\zeta=0$. Furthermore, $\widetilde{Z}^{\b,\mathrm{u}}_0(\zeta, \cdot)$ is a linear function.

{\bf \underline{Step 3}} is to  solve  $\mathbf{m}_{k,1}^\b \cdot \grad \mathrm{d}$. From the first equation of \eqref{eq:b-0-unfold}, we have
\begin{align*}
	  -\p_\zeta(\mathbf{m}_{k,1}^\b \cdot \grad \mathrm{d})
	  =\mathrm{div}_\pi (\mathbf{m}_{k,0}^\b \cdot \grad \pi).
\end{align*}
By integrating from $\zeta$ to $+\infty$, we get
\begin{equation}\label{ub-1-orth}
	\begin{aligned}
	 \mathbf{m}_{k,1}^\b \cdot \grad \mathrm{d}
		\triangleq \widetilde{Z}^\b_1(\zeta,\phi^\Int_{k,0})
	=-\frac{\chi}{\chi + \nu c_\nu} \frac{1}{c_\nu}
	\mathrm{div}_\pi (\mathbf{m}_{k,0}^\Int \cdot \grad \pi) e^{-c_\nu \zeta} ,
	\end{aligned}
\end{equation}
In particular, setting $\zeta=0$ in \eqref{ub-1-orth}, we have that on the boundary ${\partial\Omega}$,
\begin{equation} \label{ubn-1}
	  -\mathbf{m}^\b_{k,1} \cdot \mathrm{n}
	\triangleq Z^\b_1(\phi^\Int_{k,0})
	= \widetilde{Z}^\b_1(0,\phi^\Int_{k,0})
	=-\frac{\chi}{\chi + \nu c_\nu} \frac{1}{c_\nu}
	\mathrm{div}_\pi(\mathbf{m}^\Int_{k,0}\!\cdot\!\grad\pi).
\end{equation}
Note that $Z^\b_1(\cdot)$ is a linear function. Similarly, in the next round, before solving $\mathbf{m}^\b_{k,1}\!\cdot\!\grad\pi$ and $\Psi^\b_{k,1}$, we need their boundary values. In other words, the terms in ansatz \eqref{ansatz} are determined by solving the ODEs from the boundary layers and operator equations from the interior {\em alternatively}.

\subsection{Order $\sqrt{\eps}$ in the interior}
The order $\sqrt{\eps}$ in  the interior part of \eqref{eigen-equation} yields:
\begin{align} \label{int-k-1}
	 \begin{cases}
	  (\A- i \lambda_{k,0})\phi^\Int_{k,1} = i \lambda_{k,1} \phi^\Int_{k,0}\,, \\
	  \mathbf{m}^\Int_{k,1}\!\cdot\!\mathrm{n} = -\mathbf{m}^\b_{k,1}\!\cdot\!\mathrm{n}
	   =Z^\b_1(\phi^\Int_{k,0})\,.
	 \end{cases}
\end{align}

Applying Lemma \ref{Solve-A} to the system \eqref{int-k-1}, from the formula \eqref{Solve-mu}, $i \lambda_{k,1}$ can be solved as
\begin{equation}\label{lambda-k-1}
	i\lambda_{k,1}=
  \int_{\partial\Omega}(\mathbf{m}^\Int_{k,1}\!\cdot\!\mathrm{n})\Psi_{k,0}\mathrm{d}\sigma_{\!x} 
  =-\int_{\partial\Omega}(\mathbf{m}^\b_{k,1}\!\cdot\!\mathrm{n})\Psi_{k,0}\mathrm{d}\sigma_{\!x}\,.
\end{equation}

From \eqref{ubn-1} and the fact that
$ \grad \Psi_{k,0}=g^{\gamma\beta}\tfrac{\p \Psi_{k,0}}{\p\pi^\beta}\tfrac{\p}{\p\pi^\gamma}$, and $\grad \pi^\alpha=g^{\alpha\delta} \tfrac{\p}{\p\pi^\delta},$
we have
\begin{equation}\nonumber
	\begin{aligned}
	\int_{\partial\Omega} \p_{\pi^\alpha}(\grad \Psi_{k,0} \cdot\grad \pi^\alpha)
			\Psi_{k,0} \, \mathrm{d}\sigma_x
	=&-\int_{\partial\Omega} g_{\gamma
			\delta}g^{\alpha\delta}g^{\beta\gamma}\tfrac{\p\Psi_{k,0}}{\p\pi^\alpha}\tfrac{\p\Psi_{k,0}}{\p\pi^\beta}\,\mathrm{d}\sigma_x\\
	=&-\int_{\partial\Omega} |\nabla_{\!\pi} \Psi_{k,0} |^2\mathrm{d}\sigma_x\,,
	\end{aligned}
\end{equation}
where $ \nabla_{\!\pi}$ is the gradient on the tangential direction of the boundary ${\partial\Omega}$. Thus the formula \eqref{lambda-k-1} reads
\begin{equation}\label{lambda-1}
  \begin{aligned}
	i\lambda^\pm_{k,1} = & \int_{\partial\Omega} Z^\b_1(\phi^\Int_{k,0})\Psi_{k,0} \d \sigma_{\!x}
	=\Lambda \int_{\partial\Omega}|\grad\Psi_{k,0}|^2 \d \sigma_{\!x}\,,
	\end{aligned}
\end{equation}
where $\Lambda =-\frac{a^2+2a+i a^2}{2(a^2+2a +2)} \sqrt\frac{\nu}{\lambda_{k,0}^3}$ with $a=\chi \sqrt\frac{2}{\nu \lambda_{k,\, 0}} |\nabla_x \mathrm{d}|$. An important property of \eqref{lambda-1} is that the real part of $i \lambda^\pm_{k,1}$ is non-positive:
\begin{equation}\label{strict-disp}
	  \mathfrak{Re}(i \lambda^\pm_{k,1}) \leq 0\,.
\end{equation}
\begin{remark}
As mentioned in \cite{DGLM99}, the strict negativity is related the famous Schiffer's conjecture. Interestingly, if in the nonisentropic case was treated as in \cite{JM15}, there would be another term in the integrand of \eqref{lambda-1} because of the additional dissipation from heat conductivity, and they have $\mathfrak{Re}(i \lambda^\pm_{k,1})< 0$, so the geometric condition related the Schiffer's conjecture as in \cite{DGLM99},also our present paper, is not needed.
\end{remark}

If the multiplicity of $i\lambda=i\lambda_{k,0}$ as the eigenvalue of $L_0=\Delta_{\!x}$ is greater than $1$, then from \eqref{compatability}, the following compatibility condition must be satisfied:
\begin{equation}\label{CC-kl0-formula}
\int_{\partial\Omega} Z^\b_1(\phi^\Int_{k,0})\Psi_{l,0}\mathrm{d}\sigma_{\!x}=0\,,\quad \mbox{for}\quad\! \lambda_{l,0}=\lambda_{k,0}\quad\mbox{with}\quad\! k \neq l\,,
\end{equation}
which reads as
\begin{equation*}\label{condition-heat}
	\Lambda \int_{\partial\Omega} |\grad\Psi_{k,0}|^2\mathrm{d}\sigma_{\!x}=0\,,
\end{equation*}
for $\lambda_{l,0}=\lambda_{k,0}$ and $ k \neq l$. We can define the quadratic form $Q_1$ and the symmetric operator $L_1$ on $\mathrm{H}_0(\lambda)$ as
\begin{equation*}\label{Q1}
  Q_1(\Psi_{k,0}, \Psi_{l.0}) = \int_{\partial\Omega} Z^\b_1(\phi^\Int_{k,0})\Psi_{l,0}\mathrm{d}\sigma_{\!x}\,,
\end{equation*}
and
\begin{equation*}\label{L1}
   L_1 \Psi_{k,0}= i \lambda_{k,1} \Psi_{k,0}\,,
\end{equation*}
which satisfies  that
\begin{equation*} \label{non-isen}
   Q_1(\Psi_{k,0},\Psi_{l,0})=\int_\Omega L_1(\Psi_{k,0})\Psi_{l,0}\,,
\end{equation*}
and the orthogonality condition \eqref{CC-kl0-formula} is
\begin{equation}\label{ortho1}
     Q_1(\Psi_{k,0},\Psi_{l,0})=0\,,\quad \mbox{if}\quad\! \Psi_{k,0}, \Psi_{l,0}\in \mathrm{H}_0(\lambda)\quad\!\mbox{and}\quad\! k\neq l\,.
\end{equation}

From Lemma \ref{Solve-A}, under these conditions, the solution to \eqref{int-k-1} is
\begin{equation*}\label{sol-int-kl-1}
\phi^\Int_{k,1} = P^\perp_0\phi^\Int_{k,1} + P_0\phi^\Int_{k,1}\,,
\end{equation*}
When $i\lambda_{k,0}$ is a simple eigenvalue of $\Delta_{\!x}$, then $P_0\phi^\Int_{k,1}\in$ Null$(\A - i \lambda_{k,\, 0})$ vanishes by \eqref{null-space}. For this case, $\phi^\Int_{k,1}=P^\perp_0\phi^\Int_{k,1} \in$ Null$(\A - i \lambda_{k,\, 0})^\perp$ is completely determined. When $i\lambda_{k,0}$ is not simple, $P_0\phi^\Int_{k,1}$ is to be determined.

Note that the system \eqref{int-k-1} is linear and the boundary data $Z^\b_1(\phi^\Int_{k,0})$ is also linear in $\phi^\Int_{k,0}$. So $P^\perp_0\phi^\Int_{k,1}$ also linearly depends on $\phi^\Int_{k,0}$. Thus we denote
\begin{equation*}\label{Z-int-1}
P^\perp_0\phi^\Int_{k,1} = Z^\Int_1(\phi^\Int_{k,0})\,,
\end{equation*}
where $Z^\Int_1(\cdot)$ is a linear function. Furthermore,
\begin{equation}\nonumber
P_0\phi^\Int_{k,1}=\sum\limits_{l\neq k, \lambda_{l,0}=\lambda_{k,0}} a_{kl,1}\phi_{l,0}\,,
\end{equation}

where
$a_{kl,1} = \langle \phi^\Int_{k,1}|\phi_{l,0}\rangle$
will be determined later. Here, we used that $a_{kk,1} = 0$.
Note that the boundary data   for
 $\mathbf{m}^\b_{k,1}\!\cdot\!\grad \mathrm{\pi}$ is determined modulo $P_0\phi^\Int_{k,1} $. Later on, we will use the notation:
\begin{equation*}
a_{kl,j}= \langle \phi^\Int_{k,j}|\phi_{l,0}\rangle\,.
\end{equation*}

\subsection{Order $\sqrt{\eps}$ of the boundary layer}
The order $\sqrt{\eps}$ in the boundary layer is
\begin{equation}\label{b-1}
   -\A^\mathrm{d} \phi^\b_{k,2} = (\A^\pi + \mathcal{D}^\b- i \lambda_{k,0}) \phi^\b_{k,1} + (\F^\b-i\lambda_{k,1}) \phi^\b_{k,0}\,.
\end{equation}

{\bf \underline{Step 1:}} As before, step 1  is to find the ODE satisfied by $\Psi^\b_{k,2}$ which is
\begin{equation}\nonumber
  -\p_\zeta\Psi^\b_{k,2}|\grad \mathrm{d}|^2= \left((\xi+2\nu)|\grad \mathrm{d}|^2\p^2_{\zeta\zeta}- i\lambda_{k,0}\right)(\mathbf{m}^\b_{k,1}\!\cdot\!\grad\mathrm{d})+ \F^{\mathbf{m}}_{d\pi}(\mathbf{m}^\b_{k,0}\!\cdot\!\grad\pi)\,.
\end{equation}
The right-hand side of the above equation is a linear operator on $\phi^\Int_{k,0}$, (noticing the notations in \eqref{sol-b-u-0} and \eqref{ub-1-orth}). Integrating from $\zeta$ to $\infty$ gives
\begin{equation*}\label{Y-2}
   \Psi^\b_{k,2}=Y^\b_2(\zeta,\phi^\Int_{k,0})\,,
\end{equation*}
where $Y^\b_2(\zeta,\cdot)$ is a linear function. Note that $\Psi^\b_{k,j} =Y^\b_j =0$ for $j=0,1$.

{\bf\underline{Step 2}} is to solve $\mathbf{m}^\b_{k,1}\!\cdot\!\grad \pi$. Using the same method which derives the ODE \eqref{ODE-mb-0}, i.e. taking the inner product of the second equation of \eqref{b-1} with $\grad \pi$, we get that $f^\b=\mathbf{m}^\b_{k,1}\!\cdot\!\grad\pi$ satisfies the ODE
\begin{align} \label{ODE-mb-1}
	\begin{cases}
	 \mathcal{L} f^\b=
	 (i\lambda_{k,1}-\F^\mathbf{m}_{\pi\pi})(\mathbf{m}^\b_{k,0}\cdot \nabla_x \pi),\\
	  \begin{aligned}
	  	(\nu \partial_{\zeta}-\chi) f^\b|_{\zeta =0}
	  	= & \nu \left[ \left( \nabla_x \mathbf{m}_{k,\, 0}^{\Int} 
	  	+\nabla_x^\top \mathbf{m}_{k,\, 0}^{\Int} \right) \cdot n \right] \cdot \nabla_x \pi\\
	  	& \qquad+\nu |\nabla_x \pi|^2 \partial_\pi (\mathbf{m}_{k,\, 0}^\b \cdot n)
	  	+\chi \mathbf{m}_{k,1}^{\Int} \cdot \nabla_x \pi,
	  	\end{aligned}\\
	  	f^\b|_{\zeta\rightarrow \infty} = 0,
	\end{cases}
\end{align}
In fact, in the boundary conditions \eqref{ODE-mb-1}$_2$, the second terms on the right-hand side vanishes in accordance with the preceding results $\mathbf{m}_{k,\, 0}^\b \cdot n = 0$ and \alert{$\mathbf{m}_{k,\, 0}^{\Int} \cdot n = 0$}. 
As a consequence, we can rewrite \eqref{ODE-mb-1}$_2$ as,
\begin{align}
   \left(\frac{\nu}{\chi} \partial_{\zeta}-1\right)f^\b |_{\zeta=0}
	= & \mathbf{m}_{k,1}^{\Int} \cdot \nabla_x \pi+\frac{\nu}{\chi} \left( \nabla_x \mathbf{m}_{k,\, 0}^{\Int} \cdot n \right) \cdot \nabla_x \pi\\  \no
	 = & \left[ (P_0 \phi^\Int_{k,1})_{\mathbf{m}} + \left( Z^\Int_1(\phi^\Int_{k,0})  \right)_{\mathbf{m}} \right] \cdot \nabla_x \pi+V_1^{\mathbf{m}} (\phi_{k,\, 0}^{\Int}),
\end{align}
where we have used the expressions
\begin{align*}
	 \mathbf{m}_{k,1}^{\Int} = (P_0\phi^\Int_{k,1})_{\mathbf{m}}+ \left( Z^\Int_1(\phi^\Int_{k,0}) \right)_{\mathbf{m}} ,
	\quad
	  V_1^{\mathbf{m}} (\phi_{k,\, 0}^{\Int})=\frac{\nu}{\chi} \left( \nabla_x \mathbf{m}_{k,\, 0}^{\Int} \cdot n \right) \cdot \nabla_x \pi ,
\end{align*}
and the super- and subscript $\mathbf{m}$ mean the $\mathbf{m}$ component of a vector. Again, the terms on the right-hand side of equation \eqref{ODE-mb-1}$_2$ are taken values on the boundary $\partial\Omega$.

Due to the linearity of the above ODE, we can decompose $\mathbf{m}^\b_{k,1}\!\cdot\!\grad \pi= f^\b_1+ f^\b_2$, where $f^b_1$ is the solution of the ODE
	\begin{align} \label{ut-1-1}
	  \begin{cases} \displaystyle
	  	\mathcal{L} f^\b_1 = 0,
	  \\[5pt] \displaystyle
	  	\left( \frac{\nu}{\chi} \partial_{\zeta} - 1 \right) f^\b_1|_{\zeta=0} = (P_0\phi^\Int_{k,1})_{\mathbf{m}}\!\cdot\!\grad\pi = \sum\limits_{l\neq k, \lambda_{l,0}=\lambda_{k,0}} a_{kl,1}\mathbf{m}^\Int_{l,0}\!\cdot\!\grad\pi\,,
	  \\[5pt] \displaystyle
	  	f^\b_1|_{\zeta\rightarrow \infty} = 0.
	  \end{cases}
	\end{align}
%
Note that this ODE is the same as \eqref{ODE-mb-0} with the linear combination of initial data, and $\widetilde{Z}^{\b,\mathbf{m}}_0(\zeta,\cdot)$ is linear, so the solution is represented as $f^\b_1 \triangleq \widetilde{Z}^{\b,\mathbf{m}}_0(\zeta,P_0\phi^\Int_{k,1})$. Besides,  $f^\b_2$ satisfies the ODE
	\begin{align} \label{ut-1-2}
	  \begin{cases} \displaystyle
	  	\mathcal{L} f^\b_2 = (i\lambda_{k,1}-\F^\mathbf{m}_{\pi\pi}) \left( \left(\widetilde{Z}^{\b,\mathbf{m}}_0(\phi^\Int_{k,0}) \right)_{\mathbf{m}} \right) ,
		\\[5pt] \displaystyle
	  	\left( \frac{\nu}{\chi} \partial_{\zeta}-1\right)f^\b_2 |_{\zeta=0}
	   = \left( Z^\Int_1(\phi^\Int_{k,0}) \right)_{\mathbf{m}}\!\cdot\!\grad\pi
	   	 +V_1^{\mathbf{m}} (\phi_{k,\, 0}^{\Int}) \,,
		\\[5pt] \displaystyle
			f^\b_2 |_{\zeta\rightarrow \infty} = 0,
	  \end{cases}
	\end{align}
%
whose solution is represented as $f^\b_2 \triangleq \widetilde{Z}^{\b,\mathbf{m}}_1(\zeta,\phi^\Int_{k,0})$, where $\widetilde{Z}^{\b,\mathbf{m}}_1(\zeta,\cdot)$ is a linear function. Thus,
\begin{equation}\nonumber
\mathbf{m}^\b_{k,1}\!\cdot\!\grad \pi = \widetilde{Z}^{\b,\mathbf{m}}_0(\zeta,P_0 \phi^\Int_{k,1})+ \widetilde{Z}^{\b,\mathbf{m}}_1(\zeta,\phi^\Int_{k,0})\,.
\end{equation}

{\bf \underline{Step 3}} is to get $\mathbf{m}^\b_{k,2}\!\cdot\!\grad\mathrm{d} \triangleq \widetilde{Z}^\b_1(\zeta,P_0\phi^\Int_{k,1}))+ \widetilde{Z}^\b_2(\zeta,\phi^\Int_{k,0})$. Then setting $\zeta=0$, we have that on the boundary ${\partial\Omega}$
\begin{equation}\label{ub-2-orth}
   -\mathbf{m}^\b_{k,2}\!\cdot\!\mathrm{n} = Z^\b_1(P_0\phi^\Int_{k,1})+ Z^\b_2(\phi^\Int_{k,0})\,.
\end{equation}

\subsection{Order $\eps$ of the interior}
The order $\eps$ of the interior part in \eqref{eigen-equation} reads
	\begin{equation}\label{int-k-2}
		\begin{cases}
			(\A- i \lambda_{k,0})\phi^\Int_{k,2} =  i \lambda_{k,1} \phi^\Int_{k,1} + (i \lambda_{k,2} - \mathcal{D}) \phi^\Int_{k,0} \,,
		\\[3pt]
		  \mathbf{m}^\Int_{k,2}\!\cdot\!\mathrm{n} = Z^\b_1(P_0\phi^\Int_{k,1})+ Z^\b_2(\phi^\Int_{k,0})\,.
		\end{cases}
	\end{equation}
Here we use again the relation $\mathbf{m}^\Int_{k,2} \cdot n =-\mathbf{m}^\b_{k,2} \cdot n $ on ${\partial\Omega}$ and \eqref{ub-2-orth}.
Taking the inner product of \eqref{int-k-2}  with $\phi^\Int_{k,0}$ gives the first solvability condition
\begin{equation*}\label{lambda-k-2}
    i\lambda_{k,2}=\int_{\partial\Omega} Z^\b_1(P_0\phi^\Int_{k,1})\Psi_{k,0}\mathrm{d}\sigma_x
    +\int_{\partial\Omega} Z^\b_2(\phi^\Int_{k,0})\Psi_{k,0}\mathrm{d}\sigma_x+\<\mathcal{D} \phi^\Int_{k,0}| \phi^\Int_{k,0} \>\,.
\end{equation*}
Because of the linearity of $Z^\b_1$ and the orthogonality condition \eqref{ortho1}, the first term in the above equation vanishes, so $i \lambda_{k,2}$ is completely determined. To solve $\phi^\Int_{k,2}$ from \eqref{int-k-2}, we consider three cases:

{\bf Case 1:} $i\lambda_{k,0}$ is a simple eigenvalue of $L_0=\Delta_{\!x}$. For this case, no orthogonality condition is needed.

{\bf Case 2:} $i\lambda_{k,1}$ is a simple eigenvalue of $L_1$. Taking  the inner product of  \eqref{int-k-2}  with $\phi^\Int_{l,0}$ gives
\begin{equation}\label{a-kl-1-perp}
    a_{kl,1}= \frac{1}{i\lambda_{k,1}- i\lambda_{l,1}} \int_{\partial\Omega} Z^\b_2(\phi^\Int_{k,0})\Psi_{l,0}\mathrm{d}\sigma_x\,.
\end{equation}
Here, we used that  $\<\mathcal{D} \phi^\Int_{k,0},\, \phi^\Int_{l,0} \>= 0 $ for $k \neq l$.  Thus $P_0\phi^\Int_{k,1}$ is completely determined. No additional conditions on H$_0(\lambda)$ rather than \eqref{ortho1} is needed.

{\bf Case 3:} $i\lambda_{k,1}$ is an eigenvalue of $L_1$ with multiplicity greater than or equal to $2$.  For this case, we need more orthogonality conditions  on
\begin{equation}\nonumber
   \mathrm{H}_1(\lambda_1)= \{ \Psi\in \mathrm{H}_0 : L_1 \Psi = i\lambda_{k,1} \Psi \}\,.
\end{equation}
These orthogonality conditions come  from the equation \eqref{int-k-2}.
 Taking  the inner product with $\phi^\Int_{l,0}$ for $l \neq k, \lambda_{l,0}=\lambda_{k,0}$ and $\lambda_{l,1}= \lambda_{k,1}$, we get
\begin{equation}\label{CC-kl1-formula}
\int_{\partial\Omega} Z^\b_2(\phi^\Int_{k,0})\Psi_{l,0}\mathrm{d}\sigma_x =0\,.
\end{equation}
We can define the quadratic form $Q_2$ and the symmetric operator $L_2$ on $\mathrm{H}_1(\lambda_1)$ as
\begin{equation*}\label{Q2}
  Q_2(\Psi_{k,0}, \Psi_{l.0}) = \int_{\partial\Omega} Z^\b_2(\phi^\Int_{k,0})\Psi_{l,0}\mathrm{d}\sigma_x + \<\mathcal{D} \phi^\Int_{k,0}| \phi^\Int_{l,0} \>\,.
\end{equation*}
and
\begin{equation}\label{L2}
   L_2 \Psi_{k,0}= i \lambda_{k,2} \Psi_{k,0}\,,
\end{equation}
which satisfy that
\begin{equation}\nonumber
   Q_2(\Psi_{k,0},\Psi_{l,0})=\int_\Omega L_2(\Psi_{k,0})\Psi_{l,0}\mathrm{d}x\,,
\end{equation}
and the orthogonality condition \eqref{CC-kl1-formula} is
\begin{equation*}\label{ortho22}
     Q_2(\Psi_{k,0},\Psi_{l,0})=0\,,\quad \mbox{if}\quad\! \Psi_{k,0}, \Psi_{l,0}\in \mathrm{H}_1(\lambda_1)\quad\!\mbox{and}\quad\! k\neq l\,.
\end{equation*}

Under these conditions, to represent the solution of the equation \eqref{int-k-2} in Null$(\A-i\lambda_{k,0})^\perp$, we decompose the equation \eqref{int-k-2} into two parts: one is the linear combination of the equation \eqref{int-k-1}, the other only includes known terms. More precisely, decompose $\phi^\Int_{k,2}= \phi^1+ \phi^2$, where $\phi^1$ satisfies the equation
\begin{equation}\label{int-k-2-1}
	\begin{cases}
	(\A-i\lambda_{k,0})\phi^1 = i \lambda_{k,1} P_0\phi^\Int_{k,1}\,,\\
	\mathbf{m}^1\!\cdot\!\mathrm{n} = Z^\b_1(P_0\phi^\Int_{k,1})\,,
	\end{cases}
\end{equation}
whose solution in Null$(\A-i\lambda_{k,0})^\perp$ is $Z^\Int_1(P_0 \phi^\Int_{k,1})$, since the equation \eqref{int-k-2-1} is just a linear combination of equation \eqref{int-k-1}. Besides $\phi^2$ satisfies the equation
\begin{equation*}\label{int-k-2-2}
  \begin{cases}
	 (\A-i\lambda_{k,0})\phi^2=(i\lambda_{k,2}-\mathcal{D})\phi^\Int_{k,0}
	 +i\lambda_{k,1} Z^\Int_1(\phi^\Int_{k,0})\,,\\
	\mathbf{m}^2\!\cdot\!\mathrm{n} = Z^\b_2(\phi^\Int_{k,0})\,,
  \end{cases}
\end{equation*}
whose solution in Null$(\A-i\lambda_{k,0})^\perp$ is {\em completely} determined, and is denoted by $Z^\Int_2(\phi^\Int_{k,0})$. In summary, the solution of the operator equation \eqref{int-k-2} is
\begin{equation*}\label{sol-int-2}
   \phi^\Int_{k,2}= P_0 \phi^\Int_{k,2} + Z^\Int_1(P_0\phi^\Int_{k,1}) + Z^\Int_2(\phi^\Int_{k,0})\,.
\end{equation*}
For {\bf \underline{case 1}}, everything is fully solved. For {\bf \underline{case 2}}, only $P_0 \phi^\Int_{k,2}$ is undetermined. While for {\bf \underline{case 3}}, $P_0\phi^\Int_{k,1}=(P_1+P^\perp_1)(\phi^\Int_{k,1})$ where $P_1$ is the orthogonal projection on $\mathrm{H}_1(\lambda_{k,1})$ i.e.
\begin{equation}\nonumber
P_1\phi^\Int_{k,1}=\sum\limits_{l\neq k, \lambda_{l,0}=\lambda_{k,0},\lambda_{l,1} =  \lambda_{k,1}  } a_{kl,1}\phi_{l,0}. \end{equation}
while $P^\perp_1$ is the orthogonal projection on $\mathrm{H}_1^\perp(\lambda_{k,1})$, i.e.
\begin{equation}\nonumber
P_1^\perp\phi^\Int_{k,1}=\sum\limits_{l\neq k, \lambda_{l,0}=\lambda_{k,0},\lambda_{l,1}\neq \lambda_{k,1}} a_{kl,1}\phi_{l,0}\,.
\end{equation}
At this stage, $P^\perp_1\phi^\Int_{k,1}$ is completely determined by \eqref{a-kl-1-perp} and $P_1 \phi^\Int_{k,1}$ will be determined later.

\subsection{Induction process for higher orders} 
\label{sub:induction_process_for_higher_order}

We can now inductively continue the process, namely go to the order $\sqrt{\eps}^{j-1}$ of the boundary layer, then the order $\sqrt{\eps}^j$ of the interior, and so on. We should do this at least till the order $N+2$ where $N$ is the precision of the error in \eqref{error}. We will not enter into the details and refer the readers to \cite{JM15}.

Roughly speaking, we can construct recursively, on each eigenspace $\mathrm{H}_0(\lambda)$ of $-\Delta_x$, a sequence of symmetric operators $\{L_q\}_{ q \in \mathbb{N} }$ in the following way: Let $L_0= - \Delta_x$, we define $L_1$ on each one of the eigenspace $\mathrm{H}_0(\lambda)$ of $L_0$ by \eqref{L1def}. Assume that the operators $L_p$ were constructed for $p \leq q-1, q \geq 2$ in such a way that each operator $L_p$ leaves invariant the eigenspaces of the operators $L_{p'}$ for $p' < p$. Now, to construct $L_q$, it is enough to construct $L_q$ on each eigenspace $\mathrm{H}_1(\lambda_1)\cap \mathrm{H}_2(\lambda_2)\cap \cdots \cap \mathrm{H}_{q-1}(\lambda_{q-1})$ which are defined in a similar way as $\mathrm{H}_0,\, \mathrm{H}_1$, where $\lambda_1, \lambda_2, \cdots, \lambda_{q-1}$ are the correspondingly {\em multiple} eigenvalues of $L_1, L_2, \cdots, L_{q-1}$, respectively. This is done by constructing a quadratic form $Q_q$ on each space $\mathrm{H}_1(\lambda_1)\cap \mathrm{H}_2(\lambda_2)\cap \cdots \cap \mathrm{H}_{q-1}(\lambda_{q-1})$ and defining $L_q$ by
 \begin{equation*}\label{orth-cond-q}
    Q_q(\Psi, \tilde{\Psi})= \int_\Omega L_q(\Psi)\tilde{\Psi}\,\mathrm{d}x\,,\quad\!\mbox{for all}\quad\! \Psi,\tilde{\Psi}\in \mathrm{H}_1(\lambda_1)\cap \mathrm{H}_2(\lambda_2)\cap \cdots \cap \mathrm{H}_{q-1}(\lambda_{q-1})\,.
 \end{equation*}
The precise construction of the quadratic form $Q_q$ is omitted here, and we mention that once there is an eigenvalue $\lambda_q$ such that $\dim\mathrm{H}_q(\lambda_q)=1,$ i.e. $\lambda_q$ is a simple eigenvalue of $L_q$, then no more additional orthogonality conditions are needed and we can just take $L_{q'} = Id $    on  $\mathrm{H}_1(\lambda_1)\cap \mathrm{H}_2(\lambda_2)\cap \cdots \cap \mathrm{H}_{q}(\lambda_{q})$  for $q' \geq q+ 1 $.

Let $N\in \mathbb{N}$ be the integer that will appear in the order of the approximation scheme. The eigenvectors $\Psi_{k,0}$ for $\lambda_{k,0}=\lambda$ should be chosen in such a way that they are orthogonal to all the operators $L_n$ for $n \leq N+2$, which means that for $\Psi_{k,0}, \Psi_{l,0}\in \mathrm{H}_1(\lambda_1) \cap \mathrm{H}_2(\lambda_2)\cap \cdots \cap \mathrm{H}_{n-1}(\lambda_{n-1})$ and $l\neq k$, we have
	\begin{equation}\label{condition-n}
	  Q_n(\Psi_{k,0},\Psi_{l,0})=\int_\Omega L_n(\Psi_{k,0})\Psi_{l,0}\,\mathrm{d}x=0\,.
	\end{equation}
In fact, the condition \eqref{condition-n} is equivalent to: for $1 \leq q \leq n$,
\begin{equation}\label{orth-cond}
Q_q(\Psi_{k,0},\Psi_{l,0})=\int_\Omega L_q(\Psi_{k,0})\Psi_{l,0}\,\mathrm{d}x=0\,,\quad \text{if}\quad\! l\neq k,\  \Psi_{k,0},\, \Psi_{l,0}\in \mathrm{H}_0(\lambda)\,.
\end{equation}

However, for a given $\lambda= \lambda_{k,0}$, we may only need to construct a small number of the operators $L_j$ if after few steps all the eigenvalues become simple, namely if for some $j$ all the eigenvalues of $L_j$ are simple on the space $\mathrm{H}_1(\lambda_1)\cap\cdots \cap \mathrm{H}_{j-1}(\lambda_{j-1})$. It is clear that if the eigenvalues become simple for some  $j \leq N+2$, then the eigenfunctions $\Psi_k$ can be uniquely determined in view of the orthogonal condition \eqref{condition-n}. If the process does not end, then we just need to satisfy the condition till the order $N+2$ which yields a non-unique choice of eigenfunctions. In this case, we set all the undetermined pieces of the eigenfunctions to be zero.

\subsection{Remainder and error estimate}
\label{sub:remainder_and_error_estimate}

The last step is the remainder and error estimate \eqref{error}. Actually, the induction process enables us to solve in the $j$-th round that the order $\sqrt \eps^{j-1}$ in the boundary part and the order $\sqrt \eps^j$ in the interior part, with all required information in the lower orders as known results. More precisely, the order $\sqrt \eps^{j-1}$ of the boundary part takes the following form
\begin{multline} \label{eq:b j-1}
	 -\A^\mathrm{d} \phi^\b_{k,j}
	 = (\A^\pi + \mathcal{D}^\b- i \lambda_{k,0}) \phi^\b_{k,j-1} \\
	 	+ (\F^\b-i\lambda_{k,1}) \phi^\b_{k,j-2}
	 	+ (\G^\b-i\lambda_{k,2}) \phi^\b_{k,j-3}
	  - i\sum^{j-1}_{h=3} \lambda_{k,h}\ \phi^\b_{k,j-1-h}\,,
\end{multline}
while the order $\sqrt \eps^j$ of the interior part reads,
\begin{align} \label{eq:int j}
	(\A- i \lambda_{k,0})\phi^\Int_{k,j}
	= i\lambda_{k,1} \phi^\Int_{k,j-1}
	 + (i \lambda_{k,2} - \mathcal{D}) \phi^\Int_{k,j-2}
   	+\sum^j_{h=3} i \lambda_{k,h}\ \phi^\Int_{k, j-h}\,.
\end{align}
As a consequence, the truncation ansatz \eqref{ansatz} gives the remainder term
\begin{multline}
	\mathcal{R}^\pm_{k,\eps,N}
	 =\left\{(\mathcal{A}^\pi+\mathcal{D}^\b)\phi^\b_{k,N} + \mathcal{F}^\b \phi^\b_{k,N-1} + \mathcal{G}^\b \phi^\b_{k,N-2}- \sum^N_{h=0}i\lambda_h \phi^\b_{k,N-h}\right\}\sqrt{\eps}^N \\
	 	+ \text{``higher order terms''} \,.
\end{multline}

Then it is an easy matter to infer from the construction of $\phi_{k,\, i}^\b$ that the error estimate for $\|R^\pm_{k,\eps,N}\|_{L^p(\Omega)}$, namely \eqref{error}. Indeed, all the boundary parts $\phi_{k,\, i}^\b$ depend linearly and inductively on the initial eigenfunction $\phi_{k,\, 0}^{\Int}$ (i.e., $\Psi_{k,\, 0}$ and $\nabla_x \Psi_{k,\, 0}$), which is determined by the Poisson equation supplemented with Neumann boundary conditions \eqref{Laplace}. Therefore, combining the standard elliptic regularity theory and a direct change of variables $\left( \pi(x) = (x_1,\, x_2),\, \zeta(x) = \frac{{\rm d}(x)}{\sqrt \eps} \right)$ together leads us to the desired remainder estimate, corresponding to the first result in \eqref{error}.

Since the leading order term of the error contribution $\phi_{k,\, \eps,\, N}^\pm - \phi_{k,\, 0}^\pm$ is $\phi_{k,\, 0}^\b$, the same argument implies the second result in \eqref{error}, which completes the whole proof of Proposition \ref{thm:boundary-layer}.
\end{proof}


\section{Strong Convergence of $Q\mathbf{u}^{\eps}$} 

We have known from earlier results \cite{LM} that only the incompressible part of the velocity $\mathcal{P} \mathbf{u}^\eps$ converges strongly to solutions of the incompressible Navier-Stokes equations while the fast oscillation part $\mathcal{Q} \mathbf{u}^\eps= \mathbf{u}^\eps-\mathcal{P} \mathbf{u}^\eps$  convegences weakly to $0$. Here we want to justify that the oscillation part will be damped instantaneously in the case of some certain Robin boundary condition, like in the Dirichlet boundary condition case. Consequently, the fact will lead us to the strong convergence of the compressible-incompressible limit result. As pointed out in \cite{DGLM99}, the oscillation part $\mathcal{Q} \mathbf{u}^\eps$ can be expressed on the orthogonal family $\left\{ \frac{\nabla \Psi_{k,\, 0}}{ i \lambda_{k,\, 0}} \right\}_{k \ge 0}$, that is
	\begin{align*}
	  \mathcal{Q} \mathbf{u}^\eps = \sum_{k \in \mathbb{N}} \skp{\mathcal{Q}\mathbf{u}^\eps}{\frac{\nabla \Psi_{k,\, 0}}{ i \lambda_{k,\, 0}}} \frac{\nabla \Psi_{k,\, 0}}{ i \lambda_{k,\, 0}}.
	\end{align*}
As in \cite{DGLM99}, we denote by $\mathcal{I} \subset \mathbb{N}$ the set of eigenmodes $\Psi_{k,\, 0}$ satisfying that $\mathfrak{Re} (i \lambda_{k,\, 1}^\pm) < 0$, $\mathcal{J}= \mathbb{N}\backslash \mathcal{I}$ in which $ i \lambda_{k,\, 1}^\pm = 0$ for $k \in \mathcal{J}$. Note that in the latter case $\mathbf{m}_{k,\, 0}^\pm = 0$ on $\partial \Omega$, hence there is no boundary layer is created.

We split $\mathcal{Q} \mathbf{u}^\eps$ into two parts as follows,
	\begin{align*}
		\mathcal{Q}_1 \mathbf{u}^\eps = \sum_{k \in \mathcal{I}} \skp{\mathcal{Q}\mathbf{u}^\eps}{\frac{\nabla \Psi_{k,\, 0}}{ i \lambda_{k,\, 0}}} \frac{\nabla \Psi_{k,\, 0}}{ i \lambda_{k,\, 0}},
	\quad \text{ and } \quad
		\mathcal{Q}_2 \mathbf{u}^\eps = \sum_{k \in \mathcal{J}} \skp{\mathcal{Q}\mathbf{u}^\eps}{\frac{\nabla \Psi_{k,\, 0}}{ i \lambda_{k,\, 0}}} \frac{\nabla \Psi_{k,\, 0}}{ i \lambda_{k,\, 0}}.
	\end{align*}

Since it had been proved in \cite{DGLM99} that $\text{curl div}(\mathcal{Q}_2 \mathbf{u}^\eps \otimes \mathcal{Q}_2 \mathbf{u}^\eps) \to 0$ in the sense of distributions when $\mathcal{J} \neq \varnothing$, we here just need to prove that the oscillation part $\mathcal{Q}_1 \mathbf{u}^\eps$ corresponding to the damped modes converges strongly to 0, more precisely,
\begin{proposition} \label{prop:convergence of Q}
  In the case of Robin boundary conditions \eqref{eq:Robin BC}, it holds that
  	\begin{align*}
  	  \mathcal{Q}_1 \mathbf{u}^\eps \to 0 \  \text{ in } L^2(0,\, T; L^2(\Omega)), \quad \text{as } \eps \to 0.
  	\end{align*}
\end{proposition}

\begin{proof}
First, we reduce the summation with infinitely many terms to a finite number of modes, by observing the facts $\lambda_N \to \infty$ as $N \to \infty$ and
	\begin{align}
	  \sum_{k > N} \int_0^T \abs{ \skp{\mathcal{Q}_1 \mathbf{u}^\eps}{ \frac{\nabla \Psi_{k,\, 0}}{i \lambda_{k,\, 0}} } }^2 \d t \leq \frac{C}{\lambda_{N+1}^2} \abss{\nabla \mathbf{u}^\eps}_{L^2(0,\, T; L^2(\Omega))}.
	\end{align}
So it suffices to prove the that for any fixed $k$, $\skpt{\mathcal{Q}_1 \mathbf{u}^\eps}{\mathbf{m}_{k,\, 0}^\pm} \to 0$ in $L^2(0,\, T)$.

Notice the identity $\mathcal{Q} \mathbf{u}^\eps = \mathcal{Q} \mathbf{m}^\eps - \eps \mathcal{Q} (\Psi^\eps \mathbf{u}^\eps)$ and the estimate
	\begin{align}
	  \eps \abs{ \skp{ \mathcal{Q} (\Psi^\eps \mathbf{u}^\eps) }{ \nabla_x \Psi_{k,\, 0} } }
	\leq \eps \abss{\Psi^\eps}_{L^\gamma(\Omega)} \abss{\mathbf{u}^\eps}_{L^\frac{\gamma}{\gamma-1}(\Omega)}
					 \abss{\nabla_x \Psi_{k,\, 0}}_{L^\infty(\Omega)} \to 0, \quad \text{ in } L^2(0,\, T),
	\end{align}
by virtue of the relation $\gamma > \frac{d}{2}$ so that $\abss{\mathbf{u}^\eps}_{L^\frac{\gamma}{\gamma-1}(\Omega)} \leq \abss{\nabla_x \mathbf{u}^\eps}_{L^2(\Omega)}$ via the Sobolev embedding, then we only need to estimate study the quantity $\skpt{\mathcal{Q} \mathbf{m}^\eps}{\mathbf{m}_{k,\, 0}^\pm}$.

Let $\beta_{k,\, \eps}^{\pm}(t)= \skpt{\phi^\eps(t)}{\phi_{k,\, 0}^\pm}$, it is a easy matter to check
	\begin{align*}
	  2 \skp{\mathcal{Q} \mathbf{m}^\eps}{\mathbf{m}_{k,\, 0}^\pm} = \beta_{k,\, \eps}^\pm - \beta_{k,\, 0}^\mp.
	\end{align*}
So we are led to prove that
	\begin{align*}
	  b_{k,\, \eps}^{\pm} (t) = \skp{\phi^\eps(t)}{\phi_{k,\, \eps,\, 2}^{\pm}}  \to 0 \quad \text{ in } L^2(0,\, T),
	\end{align*}
due to Proposition \ref{thm:boundary-layer} applied with $N=2$,
	\begin{align}
	  \abs{ \skp{\phi^\eps(t)}{\phi_{k,\, 0}^\pm - \phi_{k,\, \eps,\, 2}^{\pm}} }
	\leq & C \sqrt\eps^{1- \frac{1}{\kappa}} \abss{\Psi^\eps}_{L^\infty(0,\, T; L^\kappa(\Omega))}
					+ C \sqrt\eps^{\frac{1}{2} - \frac{1}{2\gamma}} \abss{\mathbf{m}^\eps}_{L^\infty(0,\, T; L^\frac{2\gamma}{\gamma+1}(\Omega))}
	\no\\[2pt]
	\leq & C \sqrt\eps^\alpha \left\{ \abss{\Psi^\eps}_{L^\infty(0,\, T; L^\kappa(\Omega))} + \abss{\mathbf{m}^\eps}_{L^\infty(0,\, T; L^\frac{2\gamma}{\gamma+1}(\Omega))} \right\},
	\end{align}
here we define $\alpha = \inf(1-\frac{1}{\kappa},\, \frac{1}{2} - \frac{1}{2\gamma})$.
\end{proof}
We now state the following proposition.
\begin{proposition} \label{lemm:convergence of b-eps}
  $b_{k,\, \eps}^{\pm} (t) \to 0$ in $L^2(0,\, T)$.
\end{proposition}
\begin{proof}
Take $\phi = \phi_{k,\, \eps,\, 2}^{\pm}$ as test function in the dual system \eqref{eq:short dual form}, we then get directly
	\begin{align*}
	  \frac{\rm d}{{\rm d} t} \skp{\phi^\eps}{\phi_{k,\, \eps,\, 2}^{\pm}} - \frac{1}{\eps} \skp{\phi^\eps}{\mathcal{A}_\eps \phi_{k,\, \eps,\, 2}^{\pm}} = c^\eps(\mathbf{m}_{k,\, \eps,\, 2}^{\pm}),
	\end{align*}
combining with Proposition \ref{thm:boundary-layer}, this implies that
	\begin{align}\label{eq:ODE for b}
	  \frac{\rm d}{{\rm d} t} b_{k,\, \eps}^{\pm}(t) - \frac{\overline{i \lambda_{k,\, \eps,\, 2}^{\pm}}}{\eps} b_{k,\, \eps}^{\pm}(t) = c_{k,\, \eps}^{\pm}(t),
	\end{align}
with $c_{k,\, \eps}^{\pm} \triangleq c^\eps(\mathbf{m}_{k,\, \eps,\, 2}^{\pm}) + \eps^{-1} \skpt{\phi^\eps}{\mathcal{R}_{k,\, \eps,\, 2}^{\pm}} $.

Solving the above ODE system, we get the expression
	\begin{align}\label{eq:solu-to-b}
	  b_{k,\, \eps}^{\pm}(t) = b_{k,\, \eps}^{\pm}(0) e^{\frac{\overline{i \lambda_{k,\, \eps,\, 2}^{\pm}}}{\eps} t} + \int_0^t c_{k,\, \eps}^{\pm}(s) e^{\frac{\overline{i \lambda_{k,\, \eps,\, 2}^{\pm}}}{\eps} (t-s)} \d s.
	\end{align}
We split the following estimates into two steps.

\smallskip\noindent
\textbf{Step 1}. \textit{Estimation of the initial part}.

Recalling Proposition \ref{thm:boundary-layer}, we write
	\begin{align*}
	  i \lambda_{k,\, \eps,\, 2}^{\pm} = \pm i \lambda_{k,\, 0} + i \lambda_{k,\, 1}^\pm \sqrt\eps + i \widetilde\lambda_{k,\, 2}^\pm \eps, \text{ for some } i \widetilde\lambda_{k,\, 2}^\pm = \mathcal{O}(1),
	\end{align*}
then the expansion
	\begin{align}
	  \frac{i \lambda_{k,\, \eps,\, 2}^{\pm}}{\eps}
	 = \frac{1}{\sqrt\eps} \left[ \mathfrak{Re}(i \lambda_{k,\, 1}^\pm) +  \sqrt\eps \mathfrak{Re}(i \widetilde\lambda_{k,\, 2}^\pm) \right]
	 		- i \left[ \pm \frac{1}{\eps} \lambda_{k,\, 0} + \frac{1}{\sqrt\eps} {\rm Im}(i \lambda_{k,\, 1}^\pm) + {\rm Im}(i \widetilde\lambda_{k,\, 2}^\pm) \right]
	\end{align}
leads us to get
\begin{align}
	  &\abss{b_{k,\, \eps}^{\pm}(0) e^{\frac{\overline{i \lambda_{k,\, \eps,\, 2}^{\pm}}}{\eps} t}}_{L^2(0,\, T)}
	\\\no
	=	&\abs{b_{k,\, \eps}^{\pm}(0)} \left\{ -2 \mathfrak{Re}(i \lambda_{k,\, 1}^\pm) +  \sqrt\eps \mathfrak{Re}(i \widetilde\lambda_{k,\, 2}^\pm)  \right\}^{-\frac{1}{2}} \left\{ 1- e^{\frac{2}{\sqrt\eps} [\mathfrak{Re}(i \lambda_{k,\, 1}^\pm) +  \sqrt\eps \mathfrak{Re}(i \widetilde\lambda_{k,\, 2}^\pm)] T} \right\}^{\frac{1}{2}} \eps^{\frac{1}{4}}.
\end{align}
On the other hand,
\begin{align}
	  b_{k,\, \eps}^{\pm}(0) = & \skp{\phi^\eps}{\phi_{k,\, \eps,\, 2}^{\pm}}\\ \no
	= & \skp{\phi^\eps(0)}{\phi_{k,\, 0}^{\pm}(0)}
	+ \skp{\phi^\eps(0)}{\phi_{k,\, 0}^{\pm}(0) - \phi_{k,\, \eps,\, 2}^{\pm}(0)}.
\end{align}

Notice that the first term is bounded by some norms of $\phi^\eps(0)$, indeed,
\begin{align}
	\skp{\phi^\eps(0)}{\phi_{k,\, 0}^{\pm}(0)}
	=& \int_\Omega \Psi^{\eps,\, in} \overline{\Psi_{k,\, 0}^\pm} \d x
   +\int_\Omega \mathbf{m}^{\eps,\, in}
   \overline{\mathbf{m}_{k,\, 0}^\pm}\d x\\  \no
	\leq & \abss{\Psi^{\eps,\, in}}_{L^\gamma(\Omega)} 
	\abss{\Psi_{k,\, 0}^\pm}_{L^\frac{\gamma}{\gamma-1}(\Omega)}
 +\abss{\mathbf{m}^{\eps,\, in}}_{L^\frac{2\gamma}{\gamma+1}(\Omega)} \abss{\mathbf{m}_{k,\, 0}^\pm}_{L^\frac{2\gamma}{\gamma-1}(\Omega)}\\ \no
\leq & C \left( \abss{\Psi^{\eps,\, in}}_{L^\gamma(\Omega)}+ \abss{\mathbf{m}^{\eps,\, in}}_{L^\frac{2\gamma}{\gamma+1}(\Omega)} \right)
\end{align}

The estimation of the second term comes directly from Proposition \ref{thm:boundary-layer}, that is
\begin{align}
	  \abs{\skp{\phi^\eps(0)}{\phi_{k,\, 0}^{\pm}(0) - \phi_{k,\, \eps,\, 2}^{\pm}(0)}}
	\leq & C \sqrt\eps^{\frac{\gamma-1}{\gamma}}
	 \abss{\Psi^{\eps,\, in}}_{L^\gamma(\Omega)}
	+\sqrt\eps^{\frac{1}{2} - \frac{1}{2\gamma}} 
\abss{\mathbf{m}^{\eps,\, in}}_{L^\frac{2\gamma}{\gamma+1}(\Omega)}\\ \no
	\leq & C \eps^{\frac{\gamma-1}{2\gamma}} \left( \abss{\Psi^{\eps,\, in}}_{L^\gamma(\Omega)} + \abss{\mathbf{m}^{\eps,\, in}}_{L^\frac{2\gamma}{\gamma+1}(\Omega)} \right).
\end{align}

The above two estimations together yield the bound of $\abs{b_{k,\, \eps}^{\pm}(0)}$, and also the bound of the initial part, i.e.
\begin{align}\label{eq:esm-b-ini}
	 \abss{b_{k,\, \eps}^{\pm}(0) 
 e^{\frac{\overline{i\lambda_{k,\, \eps,\, 2}^{\pm}}}{\eps} t}}_{L^2(0,\, T)}\leq C \eps^{\frac{1}{4}}.
\end{align}

\smallskip\noindent
\textbf{Step 2}. \textit{Estimation of the remaining term}.

By the H\"older inequality, we get, for any $a(s) \in L^p(0,\, t)$, and $1\le p,\, r \le \infty$ satisfying $\frac{1}{p} + \frac{1}{r} = 1$, that
 \begin{align} \label{eq:Holder inequ for c}
   \abs{\int_0^t a(s) e^{\frac{\overline{i \lambda_{k,\, \eps,\, 2}^{\pm}}}{\eps} (t-s)} \d s}
 \le C \int_0^t a(s) e^{- \frac{1}{\sqrt\eps} \mathfrak{Re}(i \lambda_{k,\, 1}^\pm) (s-t)} \d s
 \leq C \abss{s}_{L^p(0,\, t)} \eps^{\frac{1}{2r}},
 \end{align}
where we have used the inequality,
  \begin{align*}
    \abss{e^{- \frac{1}{\sqrt\eps} \mathfrak{Re}(i \lambda_{k,\, 1}^\pm) (s-t)}}_{L^r(0,\, t)}
   = \eps^\frac{1}{2r} \left\{ \frac{1}{r \mathfrak{Re} (i \lambda_{k,\, 1}^\pm)} \left( 1- e^{-\frac{r}{\sqrt\eps} \mathfrak{Re}(i \lambda_{k,\, 1}^\pm) t} \right)  \right\}^{\frac{1}{r}} e^{\frac{1}{\sqrt\eps} \mathfrak{Re}(i \lambda_{k,\, 1}^\pm) t}.
  \end{align*}

Thus, we are now in a position to deal with the contribution of $c_{k,\, \eps}^{\pm}(t)$. We can infer from some direct calculations that
  \begin{align}
    \abs{c_{k,\, \eps}^{\pm}}
   = & \abs{ c^\eps(\mathbf{m}_{k,\, \eps,\, 2}^{\pm}) } + \eps^{-1} \abs{\skp{\phi^\eps}{\mathcal{R}_{k,\, \eps,\, 2}^{\pm}}}
  \\[2pt] \no
  \leq & c_1(t) + (\gamma-1) c_2(t) + c_3(t) + c_4(t),
  \end{align}
where
	\begin{align*}
	  c_1(t) = & \abs{ \skp{\mathbf{m}^\eps \otimes \mathbf{u}^\eps}{\nabla_x \mathbf{m}_{k,\, \eps,\, 2}^{\pm}} },
	\quad
	  c_2(t) = \abs{ \skp{\pi^\eps}{{\rm div} \mathbf{m}_{k,\, \eps,\, 2}^{\pm}} },
	\\[2pt]
	  c_3(t) = & \eps \abs{ \skp{\Psi^\eps \mathbf{u}^\eps}{{\rm div} \Sigma (\mathbf{m}_{k,\, \eps,\, 2}^{\pm})} },
	\text{ and }
	  c_4(t) = \eps^{-1} \abs{\skp{\phi^\eps}{\mathcal{R}_{k,\, \eps,\, 2}^{\pm}}}.
	\end{align*}

Estimations of the four terms follows the same process as that in \cite{DGLM99}, we here briefly give the results, as follows,
\begin{align}
	  c_1(t)
	\leq & \abss{ \mathbf{m}_{k,\, \eps,\, 2}^{\pm} }_{L^\infty(\Omega)}
			  \abss{ \mathbf{u}_1^\eps + \mathbf{u}_2^\eps}_{L^2(\Omega)}
				\abss{ \nabla_x \mathbf{u}^\eps }_{L^2(\Omega)}
	\\[2pt] \no
	  	& + \eps \abss{ \Psi^\eps }_{L^\infty(0,\, T; L^\kappa(\Omega))}
	  			 		 \abss{ \mathbf{u}^\eps }_{L^\frac{\kappa}{\kappa-1}(\Omega)}
	  			 		 \abss{ \nabla_x \mathbf{m}_{k,\, \eps,\, 2}^{\pm} }_{L^\infty(\Omega)}
	\\[2pt] \no
	\leq & C \abss{ \mathbf{u}_1^\eps }_{L^\infty(0,\, T; L^2(\Omega))}
					\abss{ \nabla_x \mathbf{u}^\eps }_{L^2(\Omega)}
			+ C \eps^\frac{1}{2} \abss{ \nabla_x \mathbf{u}^\eps }_{L^2(\Omega)}^2
			+ C \eps \abss{ \nabla_x \mathbf{u}^\eps }_{L^2(\Omega)},
	\\[2pt]
	  c_2(t)
	\leq & C \abss{ \pi^\eps }_{L^\infty(0,\, T; L^1(\Omega)) }
					\left\{ \abss{ \Psi_{k,\, \eps,\, 2}^{\pm} }_{L^\infty(\Omega)}
							 + \abss{ \mathcal{R}_{k,\, \eps,\, 2}^{\pm} }_{L^\infty(\Omega)} \right\},
	\\[2pt]
	  c_3(t)
	\leq & C \eps \abss{ D^2 \mathbf{m}_{k,\, \eps,\, 2}^{\pm} }_{L^\infty(\Omega)}
							 \abss{ \Psi^\eps }_{L^\infty(0,\, T; L^\kappa(\Omega))}
							 \abss{ \mathbf{u}^\eps }_{L^\frac{\kappa}{\kappa-1}(\Omega)}
	\leq C \eps \abss{ \nabla_x \mathbf{u}^\eps }_{L^2(\Omega)},
\end{align}
here we emphasize again that the boundary integral contribution for $\Psi^\eps \mathbf{u}^\eps$ vanishes due to the dual Navier-slip boundary conditions defined in section 2.

As for the term $c_4$, we use directly Proposition \ref{thm:boundary-layer} to get,
\begin{align}
  c_4(t) \leq&\eps^{-1} \abss{\Psi^\eps}_{L^\infty(0,\, T; L^\kappa(\Omega))}
  \abss{\widetilde\Psi_{k,\, \eps,\, 2}^{\pm}}_{L^\frac{\kappa}{\kappa-1}(\Omega)} +\eps^{-1}
  \abss{\mathbf{m}^\eps}_{L^\infty(0,\, T; L^\frac{2\gamma}{\gamma+1}(\Omega))}
  \abss{\widetilde{\mathbf{m}}_{k,\, \eps,\, 2}^{\pm}}_{L^\frac{2\gamma}{\gamma-1}(\Omega)}\no\\
  \leq & C\sqrt\eps^{1-\frac{1}{\kappa}}+C\sqrt\eps^{\frac{1}{2}-\frac{1}{2\gamma}}\no\\
  \leq &C\sqrt\eps^\alpha.
\end{align}
where we have used the notation $\mathcal{R}_{k,\, \eps,\, 2}^{\pm} = (\widetilde\Psi_{k,\, \eps,\, 2}^{\pm},\, \widetilde{\mathbf{m}}_{k,\, \eps,\, 2}^{\pm})^\top$ and $\alpha = \inf(1-\frac{1}{\kappa},\, \frac{1}{2} - \frac{1}{2\gamma})$ as before.

By applying $a(t)$ in \eqref{eq:Holder inequ for c} to $c_i$'s for $i\in \{1,\, 2,\, 3,\, 4\}$ yields the desired result of proposition \ref{lemm:convergence of b-eps},
\end{proof}
Thus it completes the proof of proposition \ref{prop:convergence of Q}. By noticing that in the case of $k \in \mathcal{J}$, ${\rm div} (\mathcal{Q}_2 \mathbf{m}^\eps \otimes \mathcal{Q}_2 \mathbf{m}^\eps)$ reduces to a gradient term, and consequently disappears in the pressure term, for details, see \cite{DGLM99}.

\appendix
\section{Some Lemmas}
In this section, we collected some results needed in the construction of the boundary layers.
\subsection{Geometry of the boundary ${\partial\Omega}$}
Next, we collect some differential geometry properties
related to the boundary ${\partial\Omega}$ which can be considered as a $(d-1)$
dimension compact Riemannian manifold with a metric induced from the
standard Euclidian metric of $\mathbb{R}^d$. Let $T({\partial\Omega})$ and $T({\partial\Omega})^\perp$
denote the tangent and normal bundles of ${\partial\Omega}$ in $\mathbb{R}^d$
respectively.

There is a tubular neighborhood $U_\delta = \{ x \in \Omega:
\mbox{dist}(x, {\partial\Omega}) < \delta\}$ of ${\partial\Omega}$ such that the nearest point
projection map is well defined and smooth. More precisely, we have the
following lemma:

\begin{lemma}\cite{JM15}\label{Projection}
If ${\partial\Omega}$ is a compact $C^k$ submanifold of dimension $d-1$ embedded
in $\mathbb{R}^d$, then there is $\delta = \delta_{\partial\Omega}>0$ and a map $\pi \in
C^{k-1}(U_\delta\,; {\partial\Omega})$
such that the following properties hold:

$(i)$ for all $x \in \Omega \subset \mathbb{R}^d$ with
$\mathrm{dist}(x\,,{\partial\Omega}) < \delta\,;$
\begin{equation*}
\pi(x) \in {\partial\Omega}\,, \quad x - \pi(x) \in T^\perp_{\pi(x)}({\partial\Omega})\,,\quad |
x - \pi(x)| = \mathrm{dist}(x\,, {\partial\Omega})\,, and
\end{equation*}
\begin{equation*}
|z - x| > \mathrm{dist}(x\,,{\partial\Omega})\quad \mbox{for any}\quad\! z \in
{\partial\Omega}\setminus\{\pi(x)\} \,;
\end{equation*}

$(ii)$
\begin{equation*}
\pi(x + z) \equiv x\,, \quad \mbox{for}\quad\! x \in {\partial\Omega}\,, z \in
T_x({\partial\Omega})^\perp\,, |z| < \delta\,,
\end{equation*}

$(iii)$ Let $\mathrm{Hess}\pi_x$ denote the Hessian of $\pi$ at
$x$, then
\begin{equation*}
\mathrm{Hess}\pi_x(V_1\,, V_2) = \mathrm{h}_x(V_1\,, V_2)\,,
\quad\mbox{for}\quad\! x\in{\partial\Omega}\quad\! V_1\,,V_2 \in T_x({\partial\Omega})\,,
\end{equation*}
where $\mathrm{h}_x$ is the second fundamental form of ${\partial\Omega}$ at $x$.
\end{lemma}
The proof of this lemma is classical, for example, see
\cite{Simon}, where the lemma is proved for general submanifold.

The viscous boundary layer we will construct has significantly different behavior over the tangential and normal directions near the boundary. This inspire us to consider the following new coordinate system, which we call the curvilinear coordinate for the tubular neighborhood $U_\delta$ defined in Lemma \ref{Projection}. Because ${\partial\Omega}$ is a $(d-1)$ dimensional manifold, so locally
$\pi(x)$ can be represented as
\begin{equation}\label{Pai}
\pi(x) = (\pi^1(x)\,, \cdots\,, \pi^{d-1}(x))\,.
\end{equation}

More precisely, the representation \eqref{Pai} could be understood in the following sense: we can introduce a new coordinate system $(\xi^1\,,\cdots\,, \xi^d)$ by
a homeomorphism which is  locally defined as $\xi: \xi(x)=(\xi'(x)\,,\xi^d(x))$ where $\xi'= (\xi^1\,,\cdots\,, \xi^{d-1})$, such that $\xi(\pi(x))=(\xi',\, 0)$ and $\mathrm{d}(x)= \xi^d$, where $\mathrm{d}(x)$ is the distance function to
the boundary ${\partial\Omega}$, i.e.
\begin{equation}\label{function-d}
\mathrm{d}(x) = \mathrm{dist}(x\,,{\partial\Omega}) = |x - \pi(x)|\,.
\end{equation}
For the simplicity of notation, we denote $``\xi'(x)=\pi(x)"$ which is the meaning of \eqref{Pai}.

It is easy to see that $\grad \mathrm{d}$ is perpendicular to the
level surface of the distance function $\mathrm{d}$, i.e. the set
$S^z = \{ x\in \Omega: \mathrm{d}(x) = z\}$. In particular, on the
boundary, $\grad \mathrm{d}$ is perpendicular to $S_0 = {\partial\Omega}$.
Without loss of generality, we can normalize the distance
function so that $\grad \mathrm{d}(x) = -\mathrm{n}(x)$ when $x \in
{\partial\Omega}$. By the definition of the projection $\pi$, we have
\begin{equation*}\label{Proj-Tangent}
\pi(x + t \grad\mathrm{d}(x)) = \pi(x)\quad \mbox{for}\quad\!\!\! t\quad\!\!\!\mbox{small}\,,
\end{equation*}
and consequently, $\grad \pi^\alpha \cdot \grad \mathrm{d} = 0$, for
$\alpha = 1\,,\cdots\,, d-1$. In particular, for $t$ small enough,
$\grad \pi^\alpha(x) \in T_x({\partial\Omega})$ when $x\in {\partial\Omega}$.

In the proof of Proposition \ref{thm:boundary-layer}, one of the key idea is that the boundary layer terms have different length scales on the tangential and normal directions of each level set $S^z$, in particular the boundary ${\partial\Omega}=S^0$. So in order to solve the ansatz, we need to project the vector fields onto tangential and normal directions by inner product with $\grad \pi^\alpha$ and $\grad \mathrm{d}$.

Next, we calculate the induced Riemannian metric from $\mathbb{R}^d$ on the family of level set
$S^z$. In a local coordinate system, these Riemannian metric can be
represented as
\begin{equation}\nonumber
g = g_{\alpha \beta} \mathrm{d}\pi^\alpha \otimes \mathrm{d}
\pi^\beta\,,
\end{equation}
where $g_{\alpha \beta}= \< \tfrac{\p}{\p{\pi^\alpha}}\,,
\tfrac{\p}{\p{\pi^\beta}}\>$. Noticing that $\tfrac{\p}{\p{x^i}} =
\tfrac{\p\pi^\alpha}{\p{x^i}}\tfrac{\p}{\p{\pi^\alpha}}$, and
$\<\tfrac{\p}{\p{x^i}}, \tfrac{\p}{\p{x^j}} \> =\delta_{ij}$, the
metric $g_{\alpha \beta}$ can be determined by
\begin{equation}\nonumber
g_{\alpha
\beta}\tfrac{\p\pi^\alpha}{\p{x^i}}\tfrac{\p\pi^\beta}{\p{x^i}}
=1\,.
\end{equation}

\subsection{The operator $\A- i\lambda^\pm_{k,0}$}
First, for any $\phi, \tilde{\phi} \in L^2(\Omega, \mathbb{C}\times \mathbb{C}^d)$, we introduce a scalar product
\begin{equation} \label{innner-product}
   \<\phi | \tilde{\phi} \> = \int_\Omega \left(\Psi \tilde{\Psi} + \mathbf{m}\!\cdot\!\tilde{\mathbf{m}} \right) \,\mathrm{d}x\,,
\end{equation}
for any $\phi=(\Psi\,,\mathbf{m})^\top $ and $\tilde{\phi}=(\tilde{\Psi}\,,\tilde{\mathbf{m}})^\top $. Under this scalar product, the eigenvectors $\phi^\pm_{k,0}$ in \eqref{phi-int-0} have norm $1$.

Now, we consider the operator $\A- i \lambda^\tau_{k,0}$, where the acoustic mode $k \geq 1$ and $\tau= +$ or $-$. This operator, especially its pseudo inverse will be important in the construction of the boundary layer. The kernel and its orthogonal under the inner product \eqref{innner-product} are
\begin{equation}\nonumber
\mbox{Null}(\A- i \lambda^\tau_{k,0}) = \mbox{Span}\{ \phi^\tau_{l,0}: \lambda_{l,0}= \lambda_{k,0}\}\,,
\end{equation}
\begin{equation}\label{null-space}
\begin{aligned}
\mbox{Null}(\A- i \lambda^\tau_{k,0})^\perp = & \mbox{Span} \{\phi^+_{l,0}\,, \phi^-_{l,0}: \lambda_{l,0}\neq \lambda_{k,0}\}\oplus \mbox{Span} \{\phi^{-\tau}_{l,0}: \lambda_{l,0}= \lambda_{k,0}\}\\ &\oplus \mbox{Null}(\A)\,.
\end{aligned}
\end{equation}
Next, we define a bounded pseudo inverse of $\mathcal{A}- i
\lambda^\tau_{k,0}$
\begin{equation}\nonumber
\big(\mathcal{A}- i \lambda^\tau_{k,0}\big)^{-1}:
\mathrm{Null}\big(\mathcal{A}- i \lambda^\tau_{k,0}\big)^\perp
\longrightarrow \mathrm{Null}\big(\mathcal{A}- i
\lambda^\tau_{k,0}\big)^\perp\,,
\end{equation}
by
\begin{equation}\label{pseudo-1}
\big(\mathcal{A}- i \lambda^\tau_{k,0}\big)^{-1}
\phi^\delta_{l,0}=\tfrac{1}{i\lambda^\delta_{l,0}- i \lambda^\tau_{k,0}}\phi^\delta_{l,0}\,,\quad\!\mbox{for any}\quad\!
\phi^\delta_{l,0}\quad\!\mbox{with}\quad\! \lambda^\delta_{l,0} \neq \lambda^\tau_{k,0}\,,
\end{equation}
and
\begin{equation}\label{pseudo-3}
\big(\mathcal{A}- i \lambda^\tau_{k,0}\big)^{-1} (\Psi, \mathbf{m})^\top
= \tfrac{1}{i \lambda^\tau_{k,0}} (\Psi, \mathbf{m})^\top \,,
\end{equation}
for any $(\Psi, \mathbf{m})^\top \in \mbox{Null}(\mathcal{A})$ and
$\tau,\delta \in \{ +,-\}$. It is obvious that this pseudo-inverse operator
is a bounded operator.

The following lemma will be used frequently in this paper.
\begin{lemma}\cite{JM15}\label{Solve-A}
For each acoustic mode $k \geq 1$ and $\tau= +$ or $-$, let $\phi^\tau_{k,0}$ be defined in \eqref{phi-int-0}, and let $\nu^\tau_k$ be a given number and $ f^\tau_k$ and $g^\tau_k$ be given vectors. Then the following system for $\phi^\tau_k = (\Psi^\tau_k\,, \mathbf{m}^\tau_k)^\top $
\begin{equation}\label{A-Lambda}
\begin{aligned}
(\A- i \lambda^\tau_{k,0})\phi^\tau_k &= i \nu^\tau_k \phi^\tau_{k,0} + f^\tau_k\,,\\
\mathbf{m}^\tau_k\!\cdot\!\mathrm{n} & = g^\tau_k\quad \mbox{on}\quad \! {\partial\Omega}\,
\end{aligned}
\end{equation}
can be solved modulo $\mbox{Null}(\A- i \lambda^\tau_{k,0})$ under the following two conditions:
\begin{itemize}
\item $i \nu^\tau_k$ satisfies
\begin{equation}\label{Solve-mu}
i \nu^\tau_k = \int_{\partial\Omega} g^\tau_k \Psi_{k,0}\mathrm{d}\sigma_{\!x} -  \langle f^\tau_k| \phi^\tau_{k,0} \rangle\,.
\end{equation}
\item If $\lambda_{k,0}$ is an eigenvalue with multiplicity greater than $1$,
 a compatibility condition is needed
\begin{equation}\label{compatability}
\int_{\partial\Omega} g^\tau_k \Psi_{l,0}\mathrm{d}\sigma_{\!x} = \langle f^\tau_k| \phi^\tau_{l,0} \rangle\,, \quad \mbox{for}\quad\! \lambda_{l,0}=\lambda_{k,0}\quad\mbox{with}\quad\! k \neq l\,.
\end{equation}
\end{itemize}
Under these two conditions, the solutions to \eqref{A-Lambda} can be uniquely represented as
\begin{equation}\label{VK}
\phi^\tau_k  = P_0 \phi^\tau_k + P_0^\perp \phi^\tau_k =\sum\limits_{\lambda_{k,0}
= \lambda_{l,0}} \< \phi^\tau_k \,|\, \phi^\tau_{l,0}\> \phi^\tau_{l,0} + P_0^\perp \phi^\tau_k \,,
\end{equation}
where $P_0^\perp \phi^\tau_k  \in \mbox{Null}\big(\mathcal{A}- i
\lambda^\tau_{k,0}\big)^\perp$ is completely determined, and $P_0 \phi^\tau_k$ is the orthogonal projection on $ \mbox{Null}\big(\mathcal{A}- i
\lambda^\tau_{k,0}\big)$ which is not determined.
\end{lemma}


\end{document}